%% file: LPNEW.tex
\documentclass[ejs]{imsart}

\arxiv{math.PR/0000000}

\RequirePackage[OT1]{fontenc}
\RequirePackage{amsthm,amsmath}
\RequirePackage{natbib}
\RequirePackage[colorlinks,citecolor=blue,urlcolor=blue]{hyperref}
\usepackage{amssymb}
\usepackage{epsf}
\usepackage{graphicx}
\usepackage{epsfig}

\arxiv{arXiv:0000.0000}

\startlocaldefs
\newcommand{\E}{\mathbf{E}}
\numberwithin{equation}{section}
\theoremstyle{plain}
\newtheorem{thm}{Theorem}[section]
\newtheorem{lemma}[thm]{Lemma}
\newtheorem{theorem}{Theorem}[section]
\newtheorem*{corollary}{Corollary}

\theoremstyle{definition}
\newtheorem{definition}{Definition}[section]

\input{symbolsdef}
\endlocaldefs

\begin{document}

\begin{frontmatter}
\title{Honest and adaptive confidence sets in $L_p$}
\runtitle{Honest and adaptive  confidence sets in $L_p$}

\begin{aug}
\author{\fnms{Alexandra} \snm{Carpentier}\ead[label=e1]{a.carpentier@statslab.cam.ac.uk}}

\runauthor{A. Carpentier}

\affiliation{University of Cambridge\thanksmark{m1}}

\address{Statistical Laboratory\\
Center for Mathematical Sciences\\
Wilberforce Road\\
CB3 0WB Cambridge\\
United Kingdom\\
\printead{e1}}

\end{aug}

\begin{abstract}
We consider the problem of constructing honest and adaptive confidence sets in $L_p$-loss (with $p\geq 1$ and $p <\infty$) over sets of Sobolev-type classes, in the setting of non-parametric Gaussian regression. The objective is to adapt the diameter of the confidence sets with respect to the smoothness degree of the underlying function, while ensuring that the true function lies in the confidence interval with high probability. When $p \geq 2$, we identify two main regimes, (i) one where adaptation is possible without any restrictions on the model, and (ii) one where critical regions have to be removed. We also prove by a matching lower bound that the size of the regions that we remove can not be chosen significantly smaller. These regimes are shown to depend in a qualitative way on the index $p$, and a continuous transition from $p=2$ to $p=\infty$ is exhibited.
\end{abstract}






\end{frontmatter}

\section{Introduction}

We consider in this paper the problem of building \textit{honest and adaptive} confidence sets around functions that belong to a $L_p$-Sobolev-type space in the non-parametric Gaussian regression setting.

This question was already investigated in $L_{\infty}$ and $L_{2}$, see for instance the papers~\citep{hoffman2002random, juditsky2003nonparametric, baraud2004confidence, robins2006adaptive,  cai2006adaptive, gine2010confidence, hoffmann2011adaptive, bull2011adaptive}. In particular, the recent papers~\citep{hoffmann2011adaptive, bull2011adaptive} develop for respectively $L_{\infty}$ and $L_2$ a minimax-optimal setting in which the construction of \textit{honest and adaptive} confidence sets is possible.

In the present paper, we extend these results to general values of $p \in [1, \infty[$. We develop minimax-optimal settings in which the construction of honest and adaptive confidence sets is possible. Since the case $1 \leq p \leq 2$ is essentially equivalent to the case $p=2$, we focus on this case $p \geq 2$ (and $p <\infty$). We prove that there is a continuous transition between the case $p=2$ described in~\citep{bull2011adaptive} and the case $p= \infty$ described in~\citep{hoffmann2011adaptive}. While the main idea of this paper, i.e.~to investigate the relationship between the problem of constructing adaptive and honest confidence sets and a certain infinite-dimensional composite testing problem, is similar to the one in~\citep{hoffmann2011adaptive, bull2011adaptive}, the techniques required for the solution of this testing problem are significantly more involved. It appears that the approaches of~\citep{hoffmann2011adaptive, bull2011adaptive} (in particular the analysis of the so-called infimum test) could not have been generalised in a straightforward way to the settings $p \in ]2, \infty[$. Also, our results imply that the curious dependence on whether $p$ is an even integer or not, that appears in a related minimax estimation problem studied in~\citep{lepski1999estimation}, is not relevant in the setting of confidence sets.

This paper is organised as follows. In Section~\ref{sec:setting}, we present the general setting. In Section~\ref{sec:results}, we provide our results, which are (i) the existence of adaptive estimators in $\|.\|_p$ norm, and (ii) the existence of honest and adaptive confidence sets in $\|.\|_p$ norm on some maximal models. The other sections of the paper present detailed proofs of these results. 
The Supplementary Material contains the proof for the existence of adaptive estimators, and also some classical preliminary results.

\section{Setting}\label{sec:setting}

Let $p\geq 1$ (and $p <\infty$). Let $\lfloor p \rfloor$ be the largest \textit{even} integer smaller than $p$ (this notation is not usual but we will need it in the course of the proofs). Let $s>t\geq 1/2$ be two degrees of smoothness.

Denote by $L_p([0,1]) = L_p$ the space of functions defined on $[0,1]$ such that $\|f\|_p^p = \int_0^1 |f(x)|^p dx < +\infty$, where $\|.\|_p$ is the usual $L_p-$norm. 

For any functions $(f,g) \in L_p \times L_q$ where $1/p + 1/q =1$, we consider the bilinear form $\langle f,g \rangle = \int_0^1 f(x) g(x) dx$.

\subsection{Wavelet basis}

We consider an orthonormal wavelet basis
\begin{equation*}
\{\phi_{k}, k \in Z_0, \psi_{l,k}, l > 0, k \in Z_l \},
\end{equation*}
such that for any integer $l\geq 0$, $Z_l \subset \mathbb Z$ and $|Z_l| \leq c2^l$ (where $c$ is a numerical constant). Also, we impose the usual conditions that for any $l>0, k\in Z_l$, $\int_0^1 \psi_{l,k}(x) dx =0$, and that $\psi_{l,k}(x) = 2^{l/2} \psi_k(2^l x)$. Such a basis exists (for instance the Cohen-Daubechies-Vial basis satisfies all these conditions, see~\cite{daub1993}).

We assume that the wavelet basis we consider satisfies the following assumption, which is quite standard.
\begin{assumption}\label{ass:wav1}
We assume that there is a universal constant $C_p$ such that we have for all $x\in [0,1]$ and any integer $J\geq 0$ 
\begin{equation*}
\sqrt{\sum_{k \in Z_0} \phi_k^2(x) + \sum_{0 < l\leq J,k \in Z_l} \psi_{l,k}^2(x)} \leq C_p 2^{J/2}.
\end{equation*}
\end{assumption}
It holds for any wavelet basis such that the mother wavelets $\psi_k$ are uniformly bounded and have sufficiently "disjoint support", i.e.~are well spread on the domain, see~\citep{hardle1998wavelets}. In other words, for uniformly bounded mother wavelets $\psi_k$ defined on a compact, this property is necessary to ensure the conservation of the norm of signals in $L_2$. This assumption is in particular satisfied for Cohen-Daubechies-Vial wavelets with $S>0$ first null moments (where the constant $C_p$ in the definition depends on $S$), see~\cite{daub1993}.

For any function $f \in L_p$, we consider the sequence of coefficients $a(f) = a$ and  the complementary sequence of coefficients $a'(f) = a'$ as
\begin{equation*}
a_{l,k} = \int_0^1 \psi_{l,k}(x) f(x) dx = \langle \psi_{l,k},f \rangle \hspace{5mm} and \hspace{5mm} a'_k = \langle \phi_k,f \rangle.
\end{equation*}

Consider the functions $f \in L_p$ that have the representation
\begin{equation}\label{eq:sequfunc}
f =  \sum_{k\in Z_0} \phi_k \langle \phi_k,f \rangle + \sum_{l>0} \sum_{k\in Z_l} \psi_{l,k} \langle \psi_{l,k},f \rangle =   \sum_{k\in Z_0} a'_k \phi_k  +\sum_{l>0} \sum_{k\in Z_0} a_{l,k} \psi_{l,k}.
\end{equation}

We moreover write for any $J\geq 0$
\begin{equation*}
\Pi_{V_J}(f) =  \sum_{k\in Z_0} a'_k \phi_k + \sum_{0< l\leq J} \sum_{k \in Z_l} a_{l,k} \psi_{l,k},
\end{equation*}
the projection of $f$ onto $\mathrm{span}(\phi_k, k\in Z_0, \psi_{l,k}, 0< l \leq J, k \in Z_l)$. We also write
\begin{equation*}
\Pi_{W_J}(f) =  \sum_{k \in Z_J} a_{J,k} \psi_{J,k}, \hspace{5mm} and \hspace{5mm} \Pi_{W_0}(f) =  \sum_{k \in Z_0} a_{k}' \phi_{k}.
\end{equation*}
the projection of $f$ onto $\mathrm{span}(\psi_{J,k}, k \in Z_J)$ or $\mathrm{span}(\phi_k, k\in Z_0)$.

\subsection{Besov spaces}

We consider for any $h \geq 1$, $p \geq 1$ (and $p <\infty$) and $r \geq 0$ the Besov norms
\begin{equation*}
\|f\|_{r,p,h} = \left(|(\langle f,\phi_{k} \rangle)_k|_{l_p}^h  + \sum_{l>0} 2^{lh(r + 1/2-1/p)}| (\langle f,\psi_{l,k}\rangle)_k |_{l_p}^h \right)^{1/h},
\end{equation*}
where $|u|_{l_p} = (\sum_{i} u_i^p)^{1/p}$ (we extend this definition as $|u|_{l_{\infty}} = \sup_i |u_i|$ for $p = \infty$) is the sequential $l_p$ norm. We extend this definition for $h=\infty$ as
\begin{equation*}
\|f\|_{r,p,\infty} = \max\left(|(\langle f,\phi_{k} \rangle)_k|_{l_p},  \sup_{l>0} 2^{l(r + 1/2-1/p)}| (\langle f,\psi_{l,k}\rangle)_k |_{l_p} \right).
\end{equation*}
The Besov-type spaces are defined for any $h \in [1, \infty]$, $p \geq 1$ (and $p <\infty$) and $r \geq 0$ as
\begin{equation*}
B_{r,p,h} = \{f \in L^p: \|f\|_{r,p,h}<+\infty\}.
\end{equation*}
We write for a given $B>0$ the $B_{r,p,\infty}$ Besov ball of smoothness $r$ and radius $B$ as
\begin{equation*} 
\Sigma(r,B) \equiv \Sigma(r,p, B) = \{f\in B_{r,p,\infty}: \|f\|_{r,p,\infty}\leq B\}.
\end{equation*}

For regular enough wavelets (e.g.~Cohen-Daubechies-Vial wavelets with $S$ first null moments), the defined Besov spaces correspond to the functional Besov spaces (Sobolev-type spaces) up to some smoothness $S \geq s$, see~\cite{meyer1992wavelets, hardle1998wavelets}. We assume that our basis satisfies this property with $s \leq S$ where $s$ is the largest smoothness that we wish to consider in our testing problem.

The spaces $B_{r,p,\infty}$ are slightly larger than the usual $L_p-$Sobolev spaces, see~\cite{bergh1976interpolation, besov1978integral}. They are however the natural objects to consider for the construction of honest and adaptive confidence sets, since they are the largest Besov spaces where the rate $n^{-r/(2r+1)}$ is minimax-optimal for functional estimation (see Section~\ref{sec:results} for references and a precise statement of this assertion).

We will consider in this paper functions $f$ that have a smoothness larger than $1/2$, which is a common assumption for the problem of the construction of adaptive and honest confidence sets, see~\citep{bull2011adaptive}. This assumption is technical and wether or not the results in this paper could be generalised to rougher functions is an open question.

\subsection{Observation scheme}

The data is a realisation of a Gaussian process defined for any $x \in [0,1]$ and for a given $n$ as
\begin{equation*}
dY^{(n)}(x) = f(x) dx + \frac{dB_x}{\sqrt{n}},
\end{equation*}
where $(B_x)_{x\in[0,1]}$ is a standard Brownian motion, and $f \in L_2$ is the function of interest.

Let us write for any $l >0$ and $k \in Z_l$ the associated wavelet coefficients as
\begin{equation*}
\hat a_{l,k}= \langle \psi_{l,k},dY^{(n)} \rangle = \int_0^1 f(x)\psi_{l,k}(x)dx +  \frac{1}{\sqrt{n}}\int_0^1\psi_{l,k}(x)dB_x ,\hspace{2mm} \mathrm {and} \hspace{2mm} a_{l,k}= \langle \psi_{l,k},f \rangle,
\end{equation*}
and for any $k \in Z_0$ the complementary wavelet coefficients as
\begin{equation*}
\hat a_{k}= \langle \phi_{k},dY^{(n)} \rangle = \int_0^1 f(x)\phi_{k}(x)dx +  \frac{1}{\sqrt{n}}\int_0^1\phi_{k}(x)dB_x ,\hspace{2mm} \mathrm {and} \hspace{2mm} a_{k}'= \langle \phi_{k},f \rangle.
\end{equation*}
We consider the wavelet estimate of $f$:
\begin{equation*}
\hat f_n =\sum_{k\in Z_0} \hat a_k' \phi_k  +\sum_{l>0} \sum_{k\in Z_l} \hat a_{l,k} \psi_{l,k}.
\end{equation*}
Projected estimates up to frequency $J\geq 0$
\begin{equation*}
\hat f_n(J) := \Pi_{V_J} \hat f_n,
\end{equation*}
and also the estimate of $f$ at level $J$
\begin{equation*}
\Pi_{W_J}(f) = \Pi_{W_J} \hat f_n.
\end{equation*}
are usually considered.

In the sequel, we write $\Pr_f$ (respectively $\mathbb E_f$) the probability (respectively expectation) under the law of $Y^{(n)}$ when the  function underlying the data is $f$. When no confusion is likely to arise, we write simply $\Pr$ (respectively $\mathbb E$).


\section{Main Results}\label{sec:results}



In this Section, we will consider the Cohen-Daubechies-Vial wavelet basis with $S>s$ first null moments (where $s$ is the largest smoothness according to which we wish to adapt). As a matter of fact, any wavelet basis satisfying the conditions defined in Section~\ref{sec:setting} (in particular Assumption~\ref{ass:wav1} and the $S$ null moments condition) will work.

\subsection{Adaptive estimation}


We provide, when $p\geq 2$ (and $p <\infty$), a result for adaptive estimation, i.e.~that adaptive estimators for functions in $B_{r,p,\infty}$ (adaptive to the smoothness $r$) exist. The technique that we use is closely related to what is proposed in the papers~\citep{lepski1992problems, gine2009uniform,bull2011adaptive}. We do not need any assumptions on $f$ except that it is in $B_{r,p,\infty}$ for a given smoothness $r$.

\begin{theorem}\label{th:inference}
Assume that $p \geq 2$ (and $p <\infty$). There exists an adaptive estimator $\tilde f_n(dY^{(n)},p)$ such that there are two constants $u_p$ and $N_p$ that depend only on $p$ such that for every $B\geq 0$ and every $r > 0$, we have for any $n \geq N_p$ that
\begin{align*}
\sup_{f \in \Sigma(r,B)} \E_f \|\tilde f_n - f \|_p &\leq u_p  \big(B^{1/(2r+1)} + 1\big) n^{-r/(2r+1)}\\
&\equiv U_p(B) n^{-r/(2r+1)}.
\end{align*}
\end{theorem}
The proof and also construction of this estimate are in Section~\ref{sec:estimation} (it is a Corollary of Theorem~\ref{th:inferenceersatz} with $c_a=0$).

This result is minimax-optimal in $r$ and in $B$ whenever $B\geq 1$ and $r>1/p$ (see~\citep{hardle1998wavelets}). 


\subsection{Previous results regarding honest and adaptive confidence sets}

For $p\geq 1$ (and $p <\infty$), a confidence set is a random subset $C_n$ of $L_p$ that depends on the data and perhaps on some additional knowledge that is available. We define its diameter in $\|.\|_p$ norm as
\begin{equation}\label{eq:diameter}
|C_n| = \inf_{\tau\geq 0}\Big\{ \tau: \exists g \in L_p: C_n \subset \{h: \|h-g\|_p \leq \tau \}\Big\}.
\end{equation}


We define honest and adaptive confidence sets as follows.
\begin{definition}[$(L_p,\alpha)-$honest and adaptive confidence set given $\mathcal P_n$, $I$ and $B$]\label{def:adapthonest}
Let $0< t< s$, $B>0$, $\alpha>0$. Let $I$ be a subset of $[t,s]$. Let $\mathcal P_n$ be a non-empty subset included in $\Sigma(t,B)$. Let $C_n(Y^{(n)},s,t,p,B,\alpha)$ be a random subset of $L_p$. $C_n$ is called $(L_p,\alpha)-$honest and adaptive given $\mathcal P_n$, $I$ and $B$ if there exists a constant $L:=L(s,t,p,B,\alpha)$ such that for any $n \geq 0$
\begin{equation*}
 \sup_{f \in \Sigma(r,B)\bigcap \mathcal P_n} {\Pr}_f\Big(|C_n| \geq L n^{-\frac{r}{2r+1}} \Big) \leq \alpha \hspace{0.2cm} \forall r \in I, \hspace{0.2cm} and \hspace{0.2cm} \inf_{f \in \mathcal P_n} {\Pr}_f \Big(f \in C_n \Big) \geq 1-\alpha.
\end{equation*}
\end{definition}
In this definition, the set $I$ is the set of Besov indexes to which we wish to adapt. We will in this paper consider the case $I=\{s,t\}$, since it is not too involved to pass from this case to the case $[t,s]$ (see e.g.~\citep{hoffmann2011adaptive,bull2011adaptive}). The model $\mathcal P_n$ is the set of functions on which we want to build honest and adaptive confidence sets. Ideally, we would like this set to be $\Sigma(t,B)$, but it will be seen not to always be possible to consider the whole set: some functions of $\Sigma(t,B)$ that are very close to $\Sigma(s,B)$ but not in this set can be a source of problems for the existence of honest and adaptive confidence sets. In some cases, as we will explain later in this section, a subset of $\Sigma(t,B)$ has to be removed.

To the best of our knowledge, the question of building honest and adaptive confidence sets in $L_p$ for $p\geq 1$ and $p <\infty$ has only been addressed in the case $p=2$. The most recent paper on this topic is~\citep{bull2011adaptive}. Also, it is noticeable that the case $p=\infty$ has been treated in the paper~\citep{hoffmann2011adaptive}, but we are not going to present the results in this case here, since they are different in essence from the case $p \in ]1,\infty[$.

The results in the paper~\citep{bull2011adaptive}, although proved in the density estimation setting, apply as well in the Gaussian process setting (it is actually more technical to derive them in the setting of density estimation). One only needs to change slightly the test statistics used. In the case $2t \geq s$, one should use instead of the statistic in Equation 35 in~\citep{bull2011adaptive}
 $$U_n(\hat f_n) = \|\Pi_{V_{j}}\hat f_n(1) -\Pi_{V_{j}}\hat f_n(2)\|_2^2 - \sum_{l \leq j}\frac{|Z_l|}{n},$$
where $\hat f_n(1), \hat f_n(2)$ are estimates of $f$ as in the present paper but computed respectively on the first and second half of the paper, and $j$ is such that $2^j \approx n^{\frac{1}{2t+1/2}}$. When $2t <s$, one should use instead of the statistic in Equation 17 in~\citep{bull2011adaptive}
 $$T_n(g) = \|\Pi_{V_{j}}\hat f_n -g\|_2^2 - \sum_{l \leq j}\frac{|Z_l|}{n},$$
where $\hat f_n$ is an estimate of $f$ as in my paper. These statistics have similar properties than the ones in the paper~\citep{bull2011adaptive}, and similar results hold for confidence sets in this setting. 

The authors of the paper~\citep{bull2011adaptive} first prove the following result when $s\leq 2t$, i.e.~that adaptive and honest confidence sets exist in this case on $\Sigma(t,B)$ itself.
\begin{theorem}[Bull and Nickl (2013)]\label{th:bull2}
Set $p=2$. Let $1/2\leq t< s$. Assume also that $s/2 \leq t$. Let $\mathcal P_n = \Sigma(t,B)$. Let $B>0$ and $\alpha>0$. There exists a $(L_2,\alpha)-$honest and adaptive confidence set given $\mathcal P_n$, $\{s,t\}$ and $B$.
\end{theorem}
In order to build such a confidence set, the authors measure the $L_2$ distance between the data and an estimate of the function $f$ that we can think of as the orthogonal projection of $\hat f_n$ on $\Sigma(s,B)$. They then define (intuitively speaking) the confidence set as being the set of functions that are at a distance that is smaller than this estimated distance from an adaptive estimate of $f$ (as defined in Theorem~\ref{th:inference}). 

It becomes however more involved when $s>2t$ (the authors actually need some additional formalism). For $G \subset L_p$, set $\|f - G\|_p = \inf_{g \in G} \|f-g\|_p$. We define for $\rho_n \geq 0$ and $B> 0$ the sets
\begin{equation}\label{eq:tildesets}
\tilde \Sigma(t, B,\rho_n) = \tilde \Sigma(t,s,p,B,\rho_n) = \Big\{f \in \Sigma(t,B): \|f - \Sigma(s,B)\|_p \geq \rho_n \Big\}.
\end{equation}
These sets are separated away from $\Sigma(s,B)$ whenever $\rho_n >0$.  They correspond to $\Sigma(t,B) \setminus \Sigma(s,B)$ where we have removed some critical functions very close to functions in $\Sigma(s,B)$ in $\|.\|_p$ norm. 
We now remind a simplified and slightly weaker version of the main Theorem in the paper~\citep{bull2011adaptive}'s when $s>2t$ (in the paper~\citep{bull2011adaptive}, the authors actually prove a stronger result, which includes adaptation also to the radius $B$ of the Besov ball). 
\begin{theorem}[Bull and Nickl (2013)]\label{th:bull}
Set $p=2$. Let $B>0$ and $\alpha>0$, and assume that $s/2 > t \geq 1/2$. Let $\mathcal P_n = \tilde \Sigma(t,B,\rho_n) \bigcup \Sigma(s,B)$ for some $\rho_n$.
\begin{itemize}
\item Let $\rho_n = C n^{-t/(2t+1/2)}$, where $C:=C(p,B,\alpha)$ is large enough. Then there exists a $(L_2,\alpha)-$honest and adaptive confidence set given $\mathcal P_n$, $\{s,t\}$ and $B$.
\item Let $\rho_n = \upsilon n^{-t/(2t+1/2)}$, where $\upsilon:=\upsilon(p,B,\alpha)$ is small enough. Then there exists no $(L_2,\alpha)-$honest and adaptive confidence set given $\mathcal P_n$, $\{s,t\}$ and $B$.
\end{itemize}
\end{theorem} 
In order to prove this theorem, the authors consider the following testing problem:
\begin{equation}\label{eq:testhaha}
H_0: f \in \Sigma(s,B) \hspace{1cm}v.s. \hspace{1cm} H_1: f \in \Sigma(t, B, \rho_n).
\end{equation}
As explained in~\cite{ hoffman2002random, juditsky2003nonparametric, hoffmann2011adaptive, bull2011adaptive}, this problem is very related to the problem of building honest and adaptive confidence sets. In the process of constructing honest and adaptive confidence sets, the authors of the paper~\citep{bull2011adaptive} construct a test for the testing problem~\eqref{eq:testhaha} that is uniformly consistent over $\tilde \Sigma(t,B,\rho_n) \bigcup \Sigma(s,B)$ for 
\begin{equation*}
\rho_n = C \max(n^{-t/(2t+1/2)}, n^{-s/(2s+1)}),
\end{equation*}
where $C$ large enough ($\rho_n$ is of order $n^{-t/(2t+1/2)}$ whenever $s\geq 2t$). More precisely, they prove that for any $n>0$ and $\alpha>0$, there exists a test $\Psi_n$ such that
\begin{equation}\label{eq:consis2}
\sup_{f \in \Sigma(s,B)} \mathbb E_f \Psi_n + \sup_{f \in \tilde \Sigma(t,B,\rho_n)} \mathbb E_f [1-\Psi_n] \leq \alpha.
\end{equation}
They then use this result to prove Theorem~\ref{th:bull}.

An indirect consequence of Theorem~\ref{th:bull} is that the case $1 \leq p \leq 2$ is not relevant, since the following negative result applies.
\begin{proposition}\label{prop:psmall}
Let $1 \leq p \leq 2$. Let $B>0$ and $\alpha>0$, and assume that $s> 2t \geq 1/p$. Let $\mathcal P_n = \tilde \Sigma(t,B,\rho_n) \bigcup \Sigma(s,B) \subset L_2$ for some $\rho_n$.  Let $\rho_n = \upsilon n^{-t/(2t+1/2)}$, where $\upsilon:=\upsilon(p,B,\alpha)$ is small enough.  There exists no $(L_p,\alpha)-$honest and adaptive confidence set given $\mathcal P_n$, $\{s,t\}$ and $B$.
\end{proposition}
The proof is the same in~\citep{bull2011adaptive} for the $L_2$ case, since the $L_2$ norms of the functions constructed to prove the impossibility result in $L_2$ are equal to the $L_p$ norm of these same functions. Since the $L_2$ norm dominates all $L_p$ norms for $p \leq 2$, Proposition~\ref{prop:psmall} implies in particular that the the confidence sets, and the regions one has to remove in order to be able to build honest and adaptive confidence sets, are \textit{both} larger in $L_p$ than in $L_2$, noting that $f \in L^2$ is a natural assumption in Gaussian white noise. It is thus preferable to use for $L_p$ the same confidence sets as for $L_2$ when $p \leq 2$. The case $p\geq 2$ (and $p <\infty$) is more interesting and is the setting we shall consider more in depth. 

Concerning lower bounds for general $p\geq 1$, the papers~\citep{ingster1987minimax, ingster1993asymptotically,ingster2002nonparametric} state that for a simpler but related testing problem
\begin{equation*}
H_0: f=0 \hspace{5mm} vs. \hspace{5mm} H_1: f \in \{f\in \Sigma(t,B): \|f - 0\|_p \geq \rho_n \},
\end{equation*}
the minimax rate of separation is 
\begin{equation}\label{hhjh}
\rho_n \geq Dn^{-\frac{t}{2t+1-1/p}},
\end{equation}
for some $D>0$, which coincides with the size of the region one has to remove in Theorem~\ref{th:bull} for $p=2$. One can wonder if it is still the minimax rate in the composite problem. The minimax order of the separation $\rho_n$ in Theorem~\ref{th:bull} is related to the  results of paper~\citep{lepski1999estimation}, where the authors prove that whenever  $p$ is an \textit{even} integer, it is possible to construct an estimate of $\|f\|_p$ whose error is of order $n^{-\frac{t}{2t+1-1/p}}$. In $L_2$, as proved in the paper~\citep{bull2011adaptive}, empirical process theory combined with this idea imply that it is possible to estimate $\|f-\Sigma(s,B)\|_2$ with precision $\max(n^{-\frac{t}{2t+1/2}}, n^{-\frac{s}{2s+1}})$. As in the paper~\citep{bull2011adaptive}, this estimate can be used to test whether $f$ is in $\Sigma(s,B)$ or not, and then to construct honest and adaptive confidence sets. However, \cite{lepski1999estimation} also provide negative results whenever $p$ is not an even integer and prove that in this case, it is not possible to estimate $\|f\|_p$ at a better rate than $(n \log(n))^{-\frac{t}{2t+1}}$. One might worry that this poor rate has repercussions on the construction of confidence sets.  By developping a new technique with respect to what was achieved in the paper~\citep{bull2011adaptive}, we will prove that it is not the case.








\subsection{Honest and adaptive confidence sets}\label{ss:exist}

We first state our results in terms of the testing problem~\eqref{eq:testhaha}, and then apply these results to the construction of honest and adaptive confidence sets. As a matter of fact, the results we provide for the existence of honest and adaptive confidence sets are a direct consequence of the solution to the testing problem~\eqref{eq:testhaha}. 


\subsubsection{Testing bounds on the related testing problem}

We state the following theorem for the existence of a uniformly consistent test (in the sense of Equation~\eqref{eq:consis2}) for the testing problem~\eqref{eq:testhaha}.
\begin{theorem}\label{th:test}
Let $B>0$, $\alpha>0$ and $n>0$.
\begin{itemize}
\item Let $\rho_n = C(B+1) \max(n^{-t/(2t+1-1/p)}, n^{-s/(2s+1)})$, where $C:=C(p,\alpha)$ is large enough. Then there exists a $(L_p,\alpha)-$uniformly consistent test (in the sense of Equation~\eqref{eq:consis2}) for the testing problem~\eqref{eq:testhaha}.
\item Let $\rho_n = \upsilon n^{-t/(2t+1-1/p)}$, where $\upsilon:=\upsilon(p,B,\alpha)$ is small enough.  Then there exists no $(L_p,\alpha)-$uniformly consistent test (in the sense of Equation~\eqref{eq:consis2}) for the testing problem~\eqref{eq:testhaha}.
\end{itemize}
\end{theorem}
Theorem~\ref{th:test} is proven in Section~\ref{sec:adaptiv1} (upper bound), and in Section~\ref{sec:lowbound} (lower bound).

There is a gap between the upper and lower bound for $\rho_n$, which is not relevant, as we will see in next paragraph, for the existence of honest and adaptive confidence sets, but which matters for the testing problem. As a matter of fact, this gap does not exist whenever $p$ is an even integer and it is possible to prove that a consistent test with 
\begin{equation*}
\rho_n \geq D n^{-\frac{t}{2t+1 - 1/p}},
\end{equation*}
exists (see~\citep{carpentier2013testing} for $p=2$, and the results can be extended for any even integer $p$). The case when $p$ is not an even integer is more involved and we conjecture that it is \textit{not} possible to build a consistent test whenever 
\begin{equation*}
\rho_n \leq D'  \max(n^{-\frac{s}{2s+1}}, n^{-\frac{t}{2t+1 - 1/p}}),
\end{equation*}
for $D'$ small enough. This remains an open question.

\subsubsection{Consequences for confidence sets}

A first consequence of Theorem~\ref{th:test} is the following theorem when $s(1-1/p) \leq t$ (analog to $s \leq 2t$ when $p=2$ in~\citep{bull2011adaptive}, Theorem~\ref{th:bull2}).
\begin{theorem}\label{th:ssmallseg1}
Let $1/2\leq t< s$. Assume also that $s(1-1/p) \leq t$. Let $\mathcal P_n = \Sigma(t,B)$. Let $B>0$ and $\alpha>0$. There exists a $(L_p,\alpha)-$honest and adaptive confidence set given $\mathcal P_n$, $\{s,t\}$ and $B$.
\end{theorem}
The proof that we provide in this paper is different from the proof in paper~\citep{bull2011adaptive}, and is in Section~\ref{sec:ssmallseg1}.
We recover the results in the paper~\citep{bull2011adaptive} for $p=2$ (Theorem~\ref{th:bull2}).

Another direct consequence of Theorem~\ref{th:test} is, in the case $s(1-1/p) > t$, the minimax-optimal order of $\rho_n$ for which the construction of confidence sets is made possible on the set $I = \{s,t\}$.
\begin{theorem}\label{th:sbig}
Let $B>0$ and $\alpha>0$, and assume that $s(1-1/p) > t \geq 1/2$. Let $\mathcal P_n = \tilde \Sigma(t,B,\rho_n) \bigcup \Sigma(s,B)$ for some $\rho_n$. 
\begin{itemize}
\item Let $\rho_n = C(B+1) n^{-t/(2t+1-1/p)}$, where $C:=C(p,\alpha)$ is large enough. Then there exists a $(L_p,\alpha)-$honest and adaptive confidence set given $\mathcal P_n$, $\{s,t\}$ and $B$.
\item Let $\rho_n = \upsilon n^{-t/(2t+1-1/p)}$, where $\upsilon:=\upsilon(p,B,\alpha)$ is small enough. Then there exists no $(L_p,\alpha)-$honest and adaptive confidence set given $\mathcal P_n$, $\{s,t\}$ and $B$.
\end{itemize}
\end{theorem}
The proof of this Theorem is in Section~\ref{s:sbigcbcb} (upper bound) and in Section~\ref{th:lbcs} (lower bound), and it is a direct consequence of Theorem~\ref{th:test}. The upper and lower bound match the results in the paper~\citep{bull2011adaptive} for $p=2$ (Theorem~\ref{th:bull}). It is also remarkable that for general $p$, the lower and upper bound also match the lower bound in Equation~\eqref{hhjh} of the simpler testing problem.

\subsection{Discussion}

\textbf{Adaptation to the smoothness in $\mathcal P_n$.} On some specific model $\mathcal P_n \subset \Sigma(t,B)$ (with $\mathcal P_n = \Sigma(t,B)$ if $s(1-1/p) \leq t$), we have created honest and adaptive confidence sets given $B$. In this paper, we considered adaptation to the exponent $r$ of the Besov spaces $B_{r,p,\infty}$, which are the natural classes to adapt to since adaptive estimation is on these spaces. We did not consider adaptation to the radius $B$ of the Besov ball. If $t < (1-1/p)s$, the model $\mathcal P_n$ on which we adapt is strictly smaller than $\Sigma(t,B)$. We however state that $\mathcal P_n$ could not have been considered significantly larger.

\medskip

\begin{figure}[htbp]
\begin{center}
\includegraphics[width = 10.5cm,height=7cm]{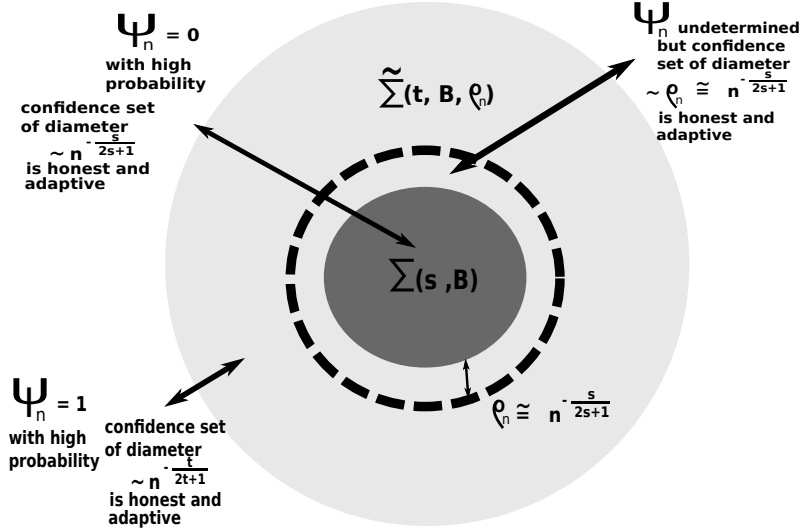}
\end{center}
\caption{Illustration of the proof of Theorem~\ref{th:ssmallseg1}.}\label{fig:1}
\end{figure}

\noindent
\textbf{Remark on the case $s(1-1/p) \leq t$.} There is a strong relation between the testing problem~\eqref{eq:testhaha}, and the problem of creating adaptive and honest confidence sets. It is remarkable however that in the case $s(1-1/p) \leq t$, although an uniformly consistent test exists only on a constrained model $\mathcal P_n = \Sigma(s,B) \bigcup \tilde \Sigma(t,B, \rho_n)$ (with $\rho_n = 2C(B+1) n^{-\frac{s}{2s+1}}$), honest and adaptive confidence sets exist on $\Sigma(t,B)$ itself (Theorem~\ref{th:ssmallseg1}). The proof of this theorem is actually very enlightening for understanding what is happening. Its nice feature is that it emphasises the connection between the testing problem~\eqref{eq:testhaha}, and the problem of building adaptive and honest confidence sets, also in the case $s(1-1/p) \leq t$ (unlike the proof in the paper~\citep{bull2011adaptive} for $p=2$ and $s \leq 2t$).  First, on $\mathcal P_n$, the existence of adaptive and honest confidence sets is a consequence of Theorem~\ref{th:test} (the proof of this fact is similar to the proof of Theorem~\ref{th:sbig}). Indeed, set for $\alpha>0$
\begin{align*}
C_n =\Big\{ f \in \Sigma(t,B): \|f - \tilde f_n\|_p \leq \frac{D}{\alpha} n^{-\frac{s}{2s+1}} (1 - \Psi_n) + \frac{D}{\alpha} n^{-\frac{t}{2t+1}} \Psi_n \Big\}
\end{align*}
where $\tilde f_n$ is the adaptive estimate constructed for Theorem~\ref{th:inferenceersatz} (which is the same as the adaptive estimate for Theorem~\ref{th:inference}), $\Psi_n$ is the test from Theorem~\ref{th:test} with level $\alpha$, and $D$ is a large enough constant depending only on $C,B$. This confidence set $C_n$ will be $\alpha-$honest and adaptive on $\mathcal P_n$ for $\{s,t\}$ given $B$, since the test $\Psi_n$ is accurate with probability at least $1-\alpha$ for any functions in $\mathcal P_n$ (since $\rho_n = 2C(B+1) n^{-\frac{s}{2s+1}}$). Second, the functions of $\Sigma(t,B) \setminus \mathcal P_n$ are at a distance smaller than $\rho_n = 2C(B+1) n^{-\frac{s}{2s+1}}$ from functions in $\Sigma(s,B)$. Theorem~\ref{th:inferenceersatz} applies to these functions, and the adaptive estimate $\tilde f_n$ is such that
\begin{equation*}
\sup_{f \in \Sigma(t,B) \setminus \mathcal P_n} \mathbb E_f \|\tilde f_n - f\|_2 \leq E n^{-\frac{s}{2s+1}},
\end{equation*}
for $E>0$ large enough and depending only on $C,B$. For this reason, mis-classifying a function $f \in \Sigma(t,B) \setminus \mathcal P_n$ into $\Sigma(s,B)$ is not problematic for confidence sets: indeed the previous equation implies that $C_n$ contains such an $f$ with probability larger than $1-\alpha$ provided that $D>E$, even though $f \not \in \Sigma(s,B)$. We illustrate the idea of the proof in Figure~\ref{fig:1}.

\noindent
\textbf{Confidence sets for a general segment $[t,s]$.} As in the paper~\citep{bull2011adaptive}, it is possible to extend Theorem~\ref{th:ssmallseg1} to the case $I = [t,s]$ (see Theorem~\ref{th:ssmallseg} in the supplementary material, Section~\ref{sec:adaptiv2}) on $\mathcal P_n = \Sigma(t,B)$, provided that $s(1-1/p) \leq t$. One can then combine the results in Theorems~\ref{th:sbig} and~\ref{th:ssmallseg} (eaxactly in the same way as in the paper~\citep{bull2011adaptive}, Theorem 5) to construct honest and adaptive confidence sets over any segment $[t,s]$, with $1/2 \leq t <s$, and on a maximal model $\mathcal P_n$. We refer the reader to Theorem 5 and it's proof in the paper~\citep{bull2011adaptive}, as the construction and proof for this fact in $L_p$ with $p \geq 2$ (and $p <\infty$) is exactly the same as what is done in this paper for $L_2$, by just combining Theorems~\ref{th:sbig} and~\ref{th:ssmallseg1}.



\medskip

\noindent
\textbf{Extension to other settings.} In the construction of confidence bands that we propose in Subsection~\ref{ss:stats2bis}, we first construct a test for the testing problem in Equation~\eqref{eq:testhaha}. In order to do that, we estimate the quantities $|a_{l,k}|^p$. The estimates we propose have good properties because the data is generated by an homocedastic Gaussian process. 
The main obstacle in more general settings is that one does not know the distribution of the noise (and in particular its $p$ first moments). Indeed, in the computation of the quantities $\hat F_p^p(l,k)$, we plug the $p$ first moments of a Gaussian distribution in order to correct the bias of $|\hat a_{l,k}|^p$ toward $|a_{l,k}|^p$. If the distribution of the noise is not Gaussian, the bias is not going to be corrected by these (Gaussian) moments, and we would want to replace them with the moments of the noise, or rather by estimates of the moments computed on the empirical residuals. A more detailed discussion can be found in the supplementary material, Appendix~\ref{app:discnongauss}.

\section{Proof of Theorem~\ref{th:test} (upper bound)}\label{sec:adaptiv1}

The method that we propose for the construction of the test in Theorem~\ref{th:test} for the testing problem~\eqref{eq:testhaha} is quite different from what was developed in~\citep{hoffmann2011adaptive, bull2011adaptive}. The main idea is to prove that for any $f\in \Sigma(t,B)$, the quantity $\|f - \Sigma(s,B)\|_p$ is close to the quantity 
\begin{equation*}
\|\Pi_{V_{j_s}}(f) - \Sigma(s,B)\|_p + \sum_{l=j_s+1}^j\|\Pi_{W_l}f\|_p.
\end{equation*}
This finding is actually very useful in practice since it allows to eliminate the empirical minimisation over $\Sigma(s,B)$ for wavelet coefficients of high resolution (which are the difficult ones to estimate), and it is a way around the technical difficulties encountered when performing the infimum test (see e.g.~papers~\citep{hoffmann2011adaptive, bull2011adaptive}). Then, one needs to estimate carefully  $\|\Pi_{V_{j_s}}(f) - \Sigma(s,B)\|_p$ and the terms $\|\Pi_{W_l}f\|_p$. The first term is easy to control using Borell's inequality. The second terms are, for each $l$, approximated by a rescaled sum of proper Taylor expansions of the terms $|\hat a_{l,k}|^p$ (i.e.~the quantities $\hat F_{p}^p(l,k)$, defined in next subsection). The variance of these estimates $\hat F_{p}^p(l,k)$ is not too difficult to bound in a proper way, since their variance is of same order than the variance of $|\hat a_{l,k}|^p$, which is bounded as $\tilde C^{(p)} \Big(n^{-p} + |a_{l,k}|^{2(p-1)}n^{-1} \Big)$. The critical quantity is the mean of these terms. When $p$ is even, the idea behind the construction of $\hat F_{p}^p(l,k)$ follows from the fact that
\begin{align*}
\mathbb E|\hat a_{l,k}|^p &= \mathbb E(a_{l,k} + \hat a_{l,k} - a_{l,k})^p\\
&=\sum_{u=0, u\hspace{2mm} even}^p \mathbf C_p^u a_{l,k}^u \frac{\E_{ G \sim \mathcal N(0,1)} \big| G  \big|^{p-u}}{n^{(p-u)/2}}.
\end{align*}
where $\mathbf C_p^u = \frac{p(p-1)...(p-u+1)}{1...u}$ is the usual binomial coefficient. $\hat F_{p}^p(l,k)$ is then $|\hat a_{l,k}|^p$ minus an unbiased estimate (constructed by induction) of the sum in last equation up to $u = p-2$. One can prove that 
\begin{equation*}
\mathbb E \hat F_{p}^p(l,k) = |a_{l,k}|^p.
\end{equation*}
Otherwise if $p$ is not an even integer, but any positive real number larger than or equal to $2$, the expectation of these Taylor expansions $\hat F_{p}^p(l,k)$ is such that
\begin{equation*}
D^{(m)} |a_{l,k}|^p \leq \mathbb E \hat F_{p}^p(l,k) \leq D^{(M)} |a_{l,k}|^{\lfloor p\rfloor}\Big( |a_{l,k}|^{p - \lfloor p\rfloor} + \frac{1}{n^{(p - \lfloor p\rfloor)/2}}\Big),
\end{equation*}
where $D^{(m)}$ and $D^{(M)}$ are two strictly positive constants. Since under $H_1$, only the lower bound on $\mathbb E \hat F_{p}^p(l,k)$ matters, and under $H_0$, only the upper bound on $\mathbb E \hat F_{p}^p(l,k)$ matters (and under $H_0$, the sum of the $|a_{l,k}|^{\lfloor p\rfloor}$ is small enough to neutralise the effect of the disturbing sum of the terms $|a_{l,k}|^{\lfloor p\rfloor}\frac{1}{n^{(p - \lfloor p\rfloor)/2}}$), we will have satisfying concentration results for the sums of $\hat F_{p}^p(l,k)$ (i.e.~$T_n(l)$). Controlling the mean and variance of these terms leads to large deviation results on the sums. This all enables us to construct an uniformly consistent test by considering if or if not these quantities (estimates of $\|\Pi_{V_{j_s}}(f) - \Sigma(s,B)\|_p$ and of the terms $\|\Pi_{W_l}f\|_p$) exceed given thresholds.

\subsection{Definition of a related testing problem}

Assume that $p\geq 2$ (and $p <\infty$) and $s>t\geq 1/2$. 


\subsection{Definition of the test statistic}\label{ss:stats2bis}

Let $0<j_s\leq j$ be two integers such that 
\begin{equation*}
j_s = \lfloor \log(n^{\frac{1}{2s+1}}) \rfloor \hspace{5mm} and \hspace{5mm} j = \lfloor \log(n^{\frac{1}{2t+1 - 1/p}}) \rfloor.
\end{equation*}

For any $u>0$ we define the following quantities 
\begin{equation*}
m_{u}^u = \E_{ G \sim \mathcal N(0,1)} \Big| G  \Big|^u.
\end{equation*}
We also define by convention, for any $l\geq j_s,k\in Z_l$, the following estimate of $(a_{l,k})^0$
\begin{align*}
 \hat F_{0}^0(l,k) =& 1.
\end{align*}
We now define by induction for any $u\geq 2$ even, the following estimates of $(a_{l,k})^u$.
\begin{align*}
\hat F_{u}^u(l,k) = \hat a_{l,k}^u - \sum_{i=0, i \hspace{1mm} even}^{u-2} \mathbf C_u^i \big(\frac{m_{u-i}}{n^{1/2}}\big)^{(u-i)} \hat F_{i}^i(l,k),
\end{align*}
where $\mathbf C_u^i = \frac{u(u-1)...(u-i+1)}{1...i}$ is the usual binomial coefficient. We extend this definition for $|a_{l,k}|^p$ if $p$ non-even (and also non necessarily integer) by setting
\begin{align*}
\hat F_{p}^p(l,k) = |\hat a_{l,k}|^{p} - \sum_{u=0, u \hspace{1mm} even}^{\lfloor p \rfloor - 2} \mathbf  C_p^u \big(\frac{m_{p-u}}{n^{1/2}}\big)^{(p-u)}  \hat F_{u}^u(l,k),
\end{align*}
where we set also for non-integer $p$ that $\mathbf C_p^u = \frac{p(p-1)...(p-u+1)}{1...u}$ (and $\mathbf C_p^0 = 1$ by convention). 

Consider the test statistics, for any $j_s\leq l \leq j$
\begin{equation*}
T_n(l) = \sum_{k \in Z_l} 2^{lp(1/2-1/p)}\hat F_{p}^p(l,k),
\end{equation*}
and also
\begin{equation*}
\tilde T_n= \inf_{g \in \Sigma(s,B)}\|\Pi_{V_{j_s}} \hat f_n - g\|_{p}.
\end{equation*}


Consider positive constants $(t_n(l))_{j_s \leq l \leq j}$ and $\tilde t_n$. We consider the test:
\begin{equation*}
 \Psi_n  = 1 -  \ind{\tilde T_n \leq \tilde t_n} \prod_{j_s \leq l\leq j} \ind{T_n(l)\leq(t_n(l))^{p}} .
\end{equation*}
We set
\begin{align*}
t_n(l) &= E_1 \Big( 2^{-ls(p-1)/p}\big(\frac{2^{l( 1- 1/p)}}{n}\big)^{1/(2p)} +2^{-\lfloor p\rfloor l s/p} \big(\frac{2^l}{n}\big)^{(p-\lfloor p\rfloor)/(2p)}\\
&+ 2^{-ls} + \sqrt{\frac{2^{(p- 1)(j+l)/(2p)}}{n}}\Big),
\end{align*}
and
\begin{equation*}
\tilde t_n = E_2 \sqrt{\frac{ 2^{j_s}}{n}},
\end{equation*}
where $E_1$ and $E_2$ are some large enough constants that depend only on $p,B$ and the desired level of the test. Then the test is uniformly consistent with 
\begin{equation*}
\rho_n = 4((B+1)C'2^{-jt} + 2\sum_{l=j_s}^j t_n(l) + 2\tilde t_n),
\end{equation*}
where $C'$ is a large enough constant that depends only on $p$ and the desired level of the test.


\subsection{Decomposition of the problem}

We justify here the test that we proposed.

\begin{lemma}\label{lem:justiftest}
Let $0<j_s<j$ be two integers. Let $(\tau_l)_{j_s\leq l\leq j}$ be a sequence of positive numbers such that $\tau_l \geq 2 \times 2^{-ls}$. Assume that $\rho_n\geq 4 C_p (B+1)(2^{-jt} + \sum_{l=j_s}^j \tau_l) $ where $C_p\geq 1$ is some positive constant that depends on $p$ only. Then we have
\begin{itemize}
 \item  $f \in \Sigma(s,B) \Rightarrow \Bigg($ $\max_{j_s \leq l\leq j} \|\Pi_{W_{l}} (f)\|_{0,p, \infty} \leq (B+1)\tau_l/2$ AND $\|\Pi_{V_{j_s}}f - \Sigma(s,B)\|_{p} =  0 $ $\Bigg)$.
 \item  $f \in \tilde \Sigma(t,B, \rho_n) \Rightarrow \Bigg($ $\max_{j_s \leq l\leq j} \|\Pi_{W_{l}} (f)\|_{0,p, \infty} \geq (B+1)\tau_l$ OR $\|\Pi_{V_{j_s}}f - \Sigma(s,B)\|_{p} \geq  3\rho_n/8 $ $\Bigg)$.
\end{itemize}

\end{lemma}
\begin{proof}
\emph{Under the null Hypothesis $H_0$}

Assume that $f \in \Sigma(s,B)$. Then $\Pi_{V_{j_s}}f \in \Sigma(s,B)$, and
\begin{align*}
\|\Pi_{V_{j_s}}f - \Sigma(s,B)\|_{p} =  0.
\end{align*}

If $f$ is in $\Sigma(s,B)$ then by definition of the Besov spaces
\begin{equation*}
\|\Pi_{V_{j}} f\|_{s,p,\infty}\leq B,
\end{equation*}
which implies by definition of the $\|.\|_{0,p,\infty}$ norm that
\begin{equation*}
\sup_{j_s \leq l\leq j} \|\Pi_{W_l} f \|_{0,p,\infty}  - \frac{B }{2^{ls}} \leq 0,
\end{equation*}
which implies that
\begin{equation*}
\sup_{j_s \leq l\leq j} \|\Pi_{W_l} f \|_{0,p,\infty}   \leq B2^{-ls} \leq \tau_l/2,
\end{equation*}

\emph{Under the alternative Hypothesis $H_1$}

Assume that $f$ is in $\tilde \Sigma(t, B,  \rho_n)$. We have by triangular inequality
\begin{align*}
\inf_{g \in \Sigma(s,B)}\|f - g\|_{p} &\leq \inf_{g \in \Sigma(s,B)}\|\Pi_{V_{j_s}} (f) - g\|_{p} + \|\Pi_{V_j \setminus V_{j_s}} (f)\|_{p} + \|\Pi_{V^-} (f)\|_{p},
\end{align*}
where we set $V^- = \mathrm{span}\big(\bigcup_{l=j+1}^{\infty}W_l\big)$. By definition of $\tilde \Sigma(t, B, \rho_n)$ and definition of the Besov spaces, we know that $\|\Pi_{V^-} f\|_{p} \leq C_pB 2^{-jt}$ (see Proposition~\ref{prop:approx}, Supplementary Material). We thus have
\begin{align*}
 \inf_{g \in \Sigma(s,B)} \|f - g\|_{p} &\leq  \inf_{g \in \Sigma(s,B)}\|\Pi_{V_{j_s}} (f) - g\|_{p} + \|\Pi_{V_j \setminus V_{j_s}} (f)\|_{p}  + C_pB2^{-jt}.
\end{align*}
We thus have (since $\rho_n \leq   \inf_{g \in \Sigma(s,B)}\|f - g\|_{p}$)
\begin{align*}
 &3\rho_n/4 \leq \rho_n - C_pB2^{-jt}  \leq  \inf_{g \in \Sigma(s,B)}\|\Pi_{V_{j_s}} (f) - g\|_{p} + \|\Pi_{V_j \setminus V_{j_s}} (f)\|_{p}.
\end{align*}
This implies that $ \inf_{g \in \Sigma(s,B)} \|\Pi_{V_{j_s}} (f) - g\|_{p} \geq 3\rho_n/8$, or $ \|\Pi_{V_j \setminus V_{j_s}} (f)\|_{p} \geq 3\rho_n/8$.

Note now that by imbrication of the Besov spaces (see Proposition~\ref{prop:embeed}) there exists a constant $C_p \geq 1$ that depends on $p$ only such that
\begin{align*}
\|\Pi_{V_{j}-V_{j_s}} (f)\|_{p}  &\leq C_p\|\Pi_{V_{j}-V_{j_s}}  (f) \|_{0,p,1}\\
&= C_p\sum_{l=j_s}^{j} \|\Pi_{W_{l}} (f)\|_{0,p, \infty}.
\end{align*}

Since $3/8 \rho_n \geq  C_p B \sum_{l=j_s}^j \tau_l$, the previous equation implies that if $\|\Pi_{V_j \setminus V_{j_s}} (f)\|_{p} \geq 3\rho_n/8$, then there exists $j_s\leq l\leq j$ such that $\|\Pi_{W_{l}} (f)\|_{0,p, \infty} \geq B\tau_l$. This concludes the proof.

\end{proof}

\subsection{Large deviations for $\|\Pi_{V_{j_s}}(\hat f_n - f)\|_{p}$}

Similarly to Lemma~\ref{lem:conclp} (Supplementary Material), we have the following Lemma.

\begin{lemma}\label{lem:conc0pp}
We have
\begin{align*}
\sup_{f \in L_p} \Pr \Big\{  \|\Pi_{V_{j_s}}(\hat f_n - f)\|_{p}   \geq (D_p + 2 C_{p/(p-1)}\sqrt{\log(1/\delta)})\sqrt{\frac{ 2^{j_s}}{n}}  \Big\} \leq \delta,
\end{align*}
where $C_{p/(p-1)}$ and $D_p$ are positive constants that depend only on $p$.
\end{lemma}
\begin{proof}
The Lemma follows directly from Propositions~\ref{prop:est} and~\ref{prop:Borell} (Supplementary Material).
\end{proof}

\subsection{Convergence tools for $T_n(l)$}

\begin{lemma}\label{lem:conc0pptopgen}
There are constants $C^{(p)}, D^{(m)}, D^{(M)}, D^{(D)}$ that depend on $p$ only such that for any $\Delta \in ]0,1[$ we have 
\begin{align}
&\Pr \Big\{\exists l: j_s \leq l \leq j, T_n(l)   \leq \|\Pi_{W_{l}}f\|_{0,p,\infty}^{p-1}\Big(D^{(m)}\|\Pi_{W_{l}}f\|_{0,p,\infty} - \sqrt{\frac{C^{(p)} D^{(D)}}{\Delta}} \sqrt{\frac{2^{l(1-1/p)}}{n}} \Big)\nonumber\\
&-  \sqrt{\frac{D^{(D)}C^{(p)}}{\Delta} \frac{2^{(p- 1)(j+l)/2}}{n^p}} \forall j_s\leq l \leq j \Big\} \leq \Delta.\label{eq:larg3gen} 
\end{align}
Also if $f \in \Sigma(s)$ then
\begin{align}
&\Pr \Big\{\exists l: j_s \leq l \leq j, T_n(l)  \geq (B^p+1)(D^{(M)} + \sqrt{\frac{C^{(p)} D^{(D)}}{\Delta}})(\tau_{p,l}(s))^p   \nonumber \\
&+  \sqrt{\frac{C^{(p)} D^{(D)}}{\Delta}} \sqrt{\frac{2^{(p- 1)(j+l)/2}}{n^p}} \forall j_s\leq l \leq j \Big\} \leq \Delta,\label{eq:larg4gen}
\end{align}
where $(\tau_{p,l}(s))^p = 2^{-ls(p-1)}\sqrt{\frac{2^{l( 1- 1/p)}}{n}} +2^{-\lfloor p\rfloor l s} \big(\frac{2^l}{n}\big)^{(p-\lfloor p\rfloor)/2} + 2^{-pls}$.

\end{lemma}

\begin{proof}

We remind that $\Pi_{W_{l}}\hat f_n = \sum_{k \in Z_l} \hat a_{l,k} \psi_{l,k}$ and that $\Pi_{W_{l}}f = \sum_{k \in Z_l} a_{l,k} \psi_{l,k}$.

\begin{lemma}\label{lem:mean}
There exists two strictly positive constants $D^{(m)}$ and $D^{(M)}$ that depend on $p$ only such that
\begin{align*}
D^{(M)}(\|\Pi_{W_{l}}f\|_{0,p,\infty}^{\lfloor p \rfloor} (\frac{2^{l}}{n})^{(p - \lfloor p \rfloor)/2} + \|\Pi_{W_{l}}f\|_{0,p,\infty}^p) \geq \E[T_n(l)] \geq  D^{(m)}  \|\Pi_{W_{l}}f\|_{0,p,\infty}^p.
\end{align*}
\end{lemma}

\begin{proof}
We provide bounds on $ \E \hat F_{p}^p(l,k)$, and this implies bounds on $T_n(l)$ by definition of $T_n(l)$.

\emph{Step 1: Expression of $ \E \hat F_{u}^u(l,k)$ for $u$ even.} Let $u>0$ be an even integer. We first prove by induction that $ \E \hat F_{u}^u(l,k) = a_{l,k}^u$.

For $u=2$ we have
\begin{align*}
 \E (\hat a_{l,k})^2 &= \E (a_{l,k} + \hat a_{l,k} - a_{l,k})^2\\
&=  a_{l,k}^2 + 2 a_{l,k} \E (\hat a_{l,k} - a_{l,k}) + \E (\hat a_{l,k} - a_{l,k})^{2}\\
&= a_{l,k}^2 + \frac{1}{n} = a_{l,k}^2 + \frac{m_{2}^2}{n},
\end{align*}
since $\hat a_{l,k} - a_{l,k} \sim \mathcal N(0, 1/n)$.This implies that $\E \hat F_{2}^2(l,k) = a_{l,k}^2$ by definition of $ \hat F_{2}^2(l,k)$. 

The induction assumption is as follows: we assume that for any $i$ even such that $2\leq i \leq u-2$, we have $\E \hat F_{i}^i(l,k) = a_{l,k}^i$. Since $u$ even, we have by a binomial expansion
\begin{align*}
 \E |\hat a_{l,k}|^u &= \E (a_{l,k} + \hat a_{l,k} - a_{l,k})^u\\
&= \sum_{i=0}^u \mathbf C_u^i a_{l,k}^i \E (\hat a_{l,k} - a_{l,k})^{u-i}\\
&= \sum_{i=0, i \hspace{1mm} even}^u \mathbf C_u^i a_{l,k}^i \frac{m_{u-i}^{u-i}}{n^{(u-i)/2}},
\end{align*}
since $\hat a_{l,k} - a_{l,k} \sim \mathcal N(0, 1/n)$ (and thus for $i$ odd, $\E (\hat a_{l,k} - a_{l,k})^{u-i} =0$). This implies by the induction assumption
\begin{align*}
 \E \hat F_{u}^u(l,k) &=  \E \hat a_{l,k}^u - \sum_{i=0, i \hspace{1mm} even}^{u-2} \mathbf C_u^i \E \hat F_{i}^i(l,k) \big(\frac{m_{u-i}}{n^{1/2}}\big)^{(u-i)}\\
&= \sum_{i=0, i \hspace{1mm} even}^u  \mathbf C_u^i   a_{l,k}^i \frac{m_{u-i}^{u-i}}{n^{(u-i)/2}} - \sum_{i=0, i \hspace{1mm} even}^{u-2} \mathbf C_u^i a_{l,k}^i \big(\frac{m_{u-i}}{n^{1/2}}\big)^{(u-i)}\\
&= \sum_{i=0, i \hspace{1mm} even}^u  \mathbf C_u^i a_{l,k}^i \big(\frac{m_{u-i}}{n^{1/2}}\big)^{(u-i)} -  \sum_{i=0, i \hspace{1mm} even}^{u-2} \mathbf C_u^i a_{l,k}^i \big(\frac{m_{u-i}}{n^{1/2}}\big)^{(u-i)}\\
&=  a_{l,k}^u.
\end{align*}
This concludes the induction.

Since for $p$ even it holds that $ \E \hat F_p^p(l,k) =  a_{l,k}^p$, it also holds in particular by definition of $T_n(l)$ and of $\|\Pi_{W_{l}}f\|_{0,p,\infty}^p$ that
\begin{equation*}
\E T_n(l) = \|\Pi_{W_{l}}f\|_{0,p,\infty}^p.
\end{equation*}

\emph{Step 2: Lower bound on $ \E \hat F_{p}^p(l,k)$ for $p$ non-even.} Consider now the case $p$ non-even.

We have since $\hat a_{l,k} - a_{l,k} \sim \mathcal N(0, 1/n)$, and since, if $a$ and $b$ in $\mathbb R$ are such that $ab>0$, then $|a+b|> \max(|a|,|b|)$,
\begin{align*}
 \E \Big|\hat a_{l,k}\Big|^p &= \E \Big|a_{l,k} + \hat a_{l,k} - a_{l,k}\Big|^p\\
&\geq |a_{l,k}|^p \Pr\big( a_{l,k} (\hat a_{l,k} - a_{l,k})\geq 0 \big) \\
&= |a_{l,k}|^p {\Pr}_{G \sim  \mathcal N(0,1)}\big(G \geq 0 \big) = |a_{l,k}|^p/2,
\end{align*}
which implies by definition of  $F_p^p(l,k)$ and also since $ \E \hat F_{u}^u(l,k) = a_{l,k}^u$ for $u$ even 
\begin{align}\label{eq:sec}
 \E \hat F_p^p(l,k) &\geq  |a_{l,k}|^p/2 - \sum_{u=0, u \hspace{1mm} even}^{\lfloor p \rfloor -2} \mathbf C_p^u a_{l,k}^u \big(\frac{m_{p-u}}{n^{1/2}}\big)^{p-u} \nonumber\\
&\geq  |a_{l,k}|^p/2 - C'\Bigg(\big(\frac{m_{p}}{n^{1/2}}\big)^{p} + a_{l,k}^{\lfloor p \rfloor} \big(\frac{m_{p-\lfloor p \rfloor}}{n^{1/2}}\big)^{(p-\lfloor p \rfloor)}\Bigg),
\end{align}
where $C'$ is some constant that depends on $p$ only.

Also, we have by a Taylor expansion since the function $h:m \rightarrow \E_{G \sim N(0,1)}|m+G|^p$ is in $C^{\infty}([0,1])$
\begin{align*}
 \E |\hat a_{l,k}|^p &= \E |a_{l,k} + \hat a_{l,k} - a_{l,k}|^p = \E_{G \sim \mathcal N(0,1)} |a_{l,k} + \frac{G}{\sqrt{n}}|^p\\
&= \sum_{u=0, u \hspace{1mm} even}^{\lfloor p \rfloor -2} \mathbf C_p^u a_{l,k}^u \big(\frac{m_{p-u}}{n^{1/2}}\big)^{p-u}+ \mathbf C_p^{\lfloor p \rfloor} a_{l,k}^{\lfloor p \rfloor} \E_{G \sim \mathcal N(0,1)} \Big| y +\frac{G}{\sqrt{n}}\Big|^{p-\lfloor p \rfloor},
\end{align*}
since $\hat a_{l,k} - a_{l,k} \sim \mathcal N(0, 1/n)$, and where $|y| \leq |a_{l,k}|$.

This implies, together with the fact that $ \E \hat F_{u}^u(l,k) = a_{l,k}^u$ for $u$ even, that
\begin{align}
 \E \hat F_p^p(l,k) &= \mathbf C_p^{\lfloor p \rfloor} a_{l,k}^{\lfloor p \rfloor} \E_{G \sim N(0,1)} \Big|y + \frac{G}{\sqrt{n}}\Big|^{p-\lfloor p \rfloor}\label{eq:thi}\\
&\geq \mathbf C_p^{\lfloor p \rfloor} a_{l,k}^{\lfloor p \rfloor} \big(\frac{m_{p-\lfloor p \rfloor}}{n^{1/2}}\big)^{p-\lfloor p \rfloor}.\label{eq:fi}
\end{align}

By considering the bound in Equation~\eqref{eq:fi} for $|a_{l,k}| \leq C''n^{-1/2}$ and the bound in Equation~\eqref{eq:sec} for $|a_{l,k}| \geq C''n^{-1/2}$, for some $C''$ that depends on $p$ only (through $m_p$, $m_{p-\lfloor p \rfloor}$ and $C'$), we obtain that there exists a constant $D^{(m)}>0$ that depends on $p$ only and such that
\begin{align*}
 \E \hat F_p^p(l,k) &\geq  D^{(m)}\max\Big(|a_{l,k}|^p,a_{l,k}^{\lfloor p \rfloor} n^{-(p-\lfloor p \rfloor)/2}\Big) \geq  D^{(m)}|a_{l,k}|^p.
\end{align*}
This leads to the lower bound in Lemma~\ref{lem:mean} by definition of $\|\Pi_{W_{l}}f\|_{0,p,\infty}^p$ and $T_n(l)$.

\emph{Step 3: Upper bound on $ \E \hat F_{p}^p(l,k)$ for $p$ non-even.}  By Equation~\eqref{eq:thi} there exists $y$ such that $|y|\leq |a_{l,k}|$ and such that
\begin{align}
 \E \hat F_p^p(l,k) &= \mathbf C_p^{\lfloor p \rfloor} a_{l,k}^{\lfloor p \rfloor} \E_{G \sim N(0,1)} \Big|y + \frac{G}{\sqrt{n}}\Big|^{p-\lfloor p \rfloor} \nonumber\\
&\leq \mathbf C_p^{\lfloor p \rfloor} a_{l,k}^{\lfloor p \rfloor}  \Bigg(4|y|^{p-\lfloor p \rfloor} + 4\E_{G \sim \mathcal N(0,1)} \Big|\frac{G}{\sqrt{n}} \Big|^{p-\lfloor p \rfloor}\Bigg) \nonumber\\
&\leq \mathbf C_p^{\lfloor p \rfloor} a_{l,k}^{\lfloor p \rfloor} \Bigg(4|y|^{p-\lfloor p \rfloor} + 4 m_{p-\lfloor p \rfloor}^{p-\lfloor p \rfloor} \Big(\frac{1}{\sqrt{n}}\Big)^{p-\lfloor p \rfloor} \Bigg) \nonumber\\
&\leq \mathbf C_p^{\lfloor p \rfloor}  \Bigg(4 |a_{l,k}|^{p} + 4 a_{l,k}^{\lfloor p \rfloor}m_{p-\lfloor p \rfloor}^{p-\lfloor p \rfloor} \Big(\frac{1}{\sqrt{n}}\Big)^{p-\lfloor p \rfloor} \Bigg). \label{eq:lastcoucou}
\end{align}
since $p-\lfloor p \rfloor \leq 2$. By H\"older's inequality,
\begin{equation*}
\sum_{k \in Z_l} a_{l,k}^{\lfloor p \rfloor} \leq |Z_l|^{\frac{p - \lfloor p \rfloor}{p}} \big(\sum_{k \in Z_l} |a_{l,k}|^{p}\big)^{\frac{\lfloor p \rfloor}{p}} \leq (c2^l)^{\frac{p - \lfloor p \rfloor}{p}} \big(\sum_{k \in Z_l} |a_{l,k}|^{p}\big)^{\lfloor p \rfloor/p},
\end{equation*}
and this implies together with Equation~\eqref{eq:lastcoucou}  and by definition of $\|\Pi_{W_{l}}f\|_{0,p,\infty}^p$ and $T_n(l)$ that
\begin{align*}
\E T_n(l) &\leq \mathbf C_p^{\lfloor p \rfloor}  \Bigg(4 \|\Pi_{W_{l}}f\|_{0,p,\infty}^p\\
&+ 4 m_{p-\lfloor p \rfloor}^{p-\lfloor p \rfloor} \Big(\frac{1}{\sqrt{n}}\Big)^{p-\lfloor p \rfloor}  2^{lp(1/2-1/p)}(c2^l)^{\frac{p - \lfloor p \rfloor}{p}} \big(\sum_{k \in Z_l} |a_{l,k}|^{p}\big)^{\lfloor p \rfloor/p} \Bigg),
\end{align*}
which leads to the upper bound in Lemma~\ref{lem:mean}.

\end{proof}

\begin{lemma}\label{lem:var}
There is a constant $C^{(p)}$ that depends on $p$ only and that is such that
\begin{align*}
\mathbb V[T_n(l)] \leq  C^{(p)} \Bigg( \sqrt{\frac{2^{l(p-1)}}{n^p}} +  \|\Pi_{W_{l}}f\|_{0,p,\infty}^{p-1}  \sqrt{\frac{2^{l(1-2/p)}}{n}} \Bigg)^2.
\end{align*}
\end{lemma}

\begin{proof}

\emph{Step 1: Bound on $\mathbb V \hat F_u^u(l,k)$ for $u$ even.} Let $u>0$ be an even integer. We prove by induction that for any such $u$, there exists a constant $C^{(u)}$ that depends on $u$ only and such that 
\begin{align*}
\mathbb V \hat F_u^u(l,k) \leq C^{(u)} ( \frac{1}{n^{u}} + \frac{|a_{l,k}|^{2(u-1)}}{n}).
\end{align*}

For $u=2$, this follows from the fact that
\begin{align*}
\mathbb V \hat F_2^2(l,k) = \mathbb V (\hat a_{l,k}^2 -1/n) = \E_{G \sim N(0,1)} \mathbb V (G^2) / n \leq C^{(2)}/n,
\end{align*}
where $C^{(2)}>0$ is a universal constant.

Assume now that it is true for any $i$ even such that $0 \leq i \leq u-2$ . We have
\begin{align}
\mathbb V \hat F_u^u(l,k) &= \mathbb V\Big[ \hat a_{l,k}^u - \sum_{i=0, i \hspace{1mm} even}^{u-2} \mathbf C_u^i\hat F_i^i(l,k) \frac{m_{u-i}^{u-i}}{n^{(u-i)/2}}\Big] \nonumber\\
&\leq u^2 \mathbb V \hat a_{l,k}^u + u^2 \sum_{i=0, i \hspace{1mm} even}^{u-2} (\mathbf C_u^i)^2 \frac{m_{u-i}^{2(u-i)}}{n^{u-i}} \mathbb V(\hat F_i^i(l,k)) \nonumber\\
&\leq u^2 \mathbb V \hat a_{l,k}^u + u^2 \sum_{i=0, i \hspace{1mm} even}^{u-2} (\mathbf C_u^i)^2 \frac{m_{u-i}^{2(u-i)}}{n^{u-i}} C^{(i)} ( \frac{1}{n^{i}} + \frac{|a_{l,k}|^{2(i-1)}}{n}). \label{eq:trutru}
\end{align}

By~\citep{ingster2002nonparametric} (page 86), if $G\sim \mathcal N(0,1)$ and $m$ is a real number, and $u\geq 2$, then there is a constant $D^{(u)}$ that depends only on $u$ such that $\mathbb V (|G+m|^{u})  \leq D^{(u)}(1 + |m|^{2u-2})$. We thus have
\begin{align}\label{eq:ing}
 \mathbb V \Big(|\hat a_{l,k}|^{u}\Big) &\leq D^{(u)} ( \frac{1}{n^{u}} + \frac{|a_{l,k}|^{2(u-1)}}{n}).
\end{align}
This implies together with Equation~\eqref{eq:trutru} that
\begin{align}
\mathbb V \hat F_u^u(l,k) &\leq u^2 D^{(u)} ( \frac{1}{n^{u}} + \frac{|a_{l,k}|^{2(u-1)}}{n}) \nonumber\\ 
&+ u^2 \sum_{i=0, i \hspace{1mm} even}^{u-2} (\mathbf C_u^i)^2 \frac{m_{u-i}^{2(u-i)}}{n^{u-i}} C^{(i)} ( \frac{1}{n^{i}} + \frac{|a_{l,k}|^{2(i-1)}}{n}) \nonumber\\
&\leq C^{(u)}  ( \frac{1}{n^{u}} + \frac{|a_{l,k}|^{2(u-1)}}{n}), \label{eq:finish}
\end{align}
where $C^{(u)}$ depends on $u$ only. This concludes the induction.

\emph{Step 2: Bound on $\mathbb V \hat F_p^p(l,k)$.} If $p$ is even, we consider the bound in Step 1. If $p$ is not even, we know similarly as in Equation~\eqref{eq:ing} that
\begin{align*}
 \mathbb V |\hat a_{l,k}|^{p} &\leq D^{(p)} ( \frac{1}{n^{p}} + \frac{|a_{l,k}|^{2(p-1)}}{n}),
\end{align*}
which implies by definition of $\hat F_p^p(l,k)$ and also since for any $u$ even, Equation~\eqref{eq:finish} holds, the following result
\begin{align}
\mathbb V \hat F_p^p(l,k) =& \mathbb V\Big[ |\hat a_{l,k}|^p - \sum_{u=0, i \hspace{1mm} even}^{\lfloor p \rfloor -2} \mathbf C_p^u\hat F_u^u(l,k) \frac{m_{p-u}^{p-u}}{n^{(p-u)/2}}\Big] \nonumber\\
\leq& (\lfloor p\rfloor +1)^2 D^{(p)} ( \frac{1}{n^{p}}+ \frac{|a_{l,k}|^{2(p-1)}}{n})\\
&+ (\lfloor p \rfloor+1)^2 \sum_{u=0, i \hspace{1mm} even}^{\lfloor p \rfloor-2} (\mathbf C_p^u)^2 \frac{m_{p-u}^{2(p-u)}}{n^{p-u}} C^{(p)} ( \frac{1}{n^{u}} + \frac{|a_{l,k}|^{2(u-1)}}{n}) \nonumber\\
\leq& C^{(p)}  ( \frac{1}{n^{p}} + \frac{|a_{l,k}|^{2(p-1)}}{n}), \label{eq:Fpp}
\end{align}
where $C^{(p)}$ depends on $p$ only.

\emph{Step 3: Conclusion}

Now by definition of $T_n(l)$ and since the $\hat a_{l,k}$ are independent, we have by Equation~\eqref{eq:Fpp}
\begin{align*}
\mathbb V T_n(l) &= \sum_{k \in Z_l} 2^{2lp(1/2 - 1/p)} \mathbb V \hat F_u^u(l,k)\\
&\leq C^{(p)}  \big( \frac{2^{lp(1 - 1/p)}}{n^{p}} +2^{lp(1 - 2/p)} \sum_{k \in Z_l}\frac{|a_{l,k}|^{2(p-1)}}{n} \big)\\
&\leq C^{(p)}  \Big( \frac{2^{lp(1 - 1/p)}}{n^{p}} + 2^{lp(1 - 2/p)} \frac{\big(\sum_{k \in Z_l} |a_{l,k}|^{p}\big)^{2(p-1)/p}}{n} \Big)\\
&= C^{(p)}  \Big( \frac{2^{lp(1 - 1/p)}}{n^{p}} + 2^{l(1 - 2/p)} \frac{\big(\sum_{k \in Z_l} 2^{lp(1/2 - 1/p)} |a_{l,k}|^{p}\big)^{2(p-1)/p}}{n} \Big)\\
&\leq C^{(p)}  \big( \frac{2^{lp(1 - 1/p)}}{n^{p}} + \frac{2^{l(1 - 2/p)}}{n} \|\Pi_{W_{l}}f\|_{0,p,\infty}^{2(p-1)} \big).
\end{align*}
since $\sum_{k \in Z_l} |a_{l,k}|^{2(p-1)} \leq  \Big(\sum_{k \in Z_l} |a_{l,k}|^p\Big)^{2(p-1)/p}$ (since $2(p-1)/p \geq 1$). This concludes the proof.
\end{proof}

Lemma~\ref{lem:var} implies by Chebyshev's inequality that for $0<\delta_l\leq 1$
\begin{align*}
\Pr \Big\{ \big| T_n(l)  - \E [ T_n(l)] \big| \geq \sqrt{\frac{1}{\delta_l}} \sqrt{C^{(p)}  \big( \frac{2^{lp(1 - 1/p)}}{n^{p}} + \frac{2^{l(1 - 2/p)}}{n} \|\Pi_{W_{l}}f\|_{0,p,\infty}^{2(p-1)} \big)} \Big\} \leq \delta_l.
\end{align*}
which implies since for any $a, b \geq 0$, we have $\sqrt{a+b} \leq \sqrt{a} + \sqrt{b}$,
\begin{align}\label{eq:cheby}
\Pr \Big\{ \big| T_n(l)  - \E [ T_n(l)] \big| \geq \sqrt{\frac{C^{(p)}}{\delta_l}} \Bigg( \sqrt{\frac{2^{l(p-1)}}{n^p}} +  \|\Pi_{W_{l}}f\|_{0,p,\infty}^{p-1}  \sqrt{\frac{2^{l(1-2/p)}}{n}}\Bigg) \Big\} \leq \delta_l.
\end{align}

Consider $\frac{1}{\delta_l} = \frac{D^{(D)}}{\Delta} \min(2^{(j-l)(p-1)/2}, 2^{l/p})$ where $0<\Delta\leq 1$ and $D^{(D)}$ is a positive constant such that $\sum_{j_s\leq l\leq j} \frac{1}{\min(2^{(j-l)(p-1)/2}, 2^{l/p})} = D^{(D)}$. $D^{(D)}$ is a constant that is bounded depending on $p$ only as a sum of the max of two geometric series. We thus have by an union bound on Equation~\eqref{eq:cheby}

\begin{align}
&\Pr \Big\{\forall j_s\leq l \leq j, \big| T_n(l)  - \E [ T_n(l)] \big| \geq \nonumber\\ 
&\sqrt{\frac{C^{(p)} D^{(D)}}{\Delta}} \Bigg( \sqrt{\frac{2^{(p-1)(j+l)/2}}{n^p}} +  \|\Pi_{W_{l}}f\|_{0,p,\infty}^{p-1}  \sqrt{\frac{2^{l(1-1/p)}}{n}}\Bigg) \Big\} \leq \sum_{j_s \leq l\leq j} \delta_l \leq \Delta. \label{eq:cheby2}
\end{align}

Equation~\eqref{eq:cheby2} implies with Lemma~\ref{lem:mean} (right hand side) that
\begin{align*}
&\Pr \Big\{\forall j_s\leq l \leq j,  T_n(l)   \leq \|\Pi_{W_{l}}f\|_{0,p,\infty}^{p-1}\Big(D^{(m)}\|\Pi_{W_{l}}f\|_{0,p,\infty} - \sqrt{\frac{C^{(p)} D^{(D)}}{\Delta}} \sqrt{\frac{2^{l(1-1/p)}}{n}} \Big)\\
&-  \sqrt{\frac{D^{(D)} C^{(p)}}{\Delta} \frac{2^{(p- 1)(j+l)/2}}{n^p}}  \Big\} \leq \Delta.
\end{align*}
Assume now that $f \in \Sigma(s)$. Then we have by Proposition~\ref{prop:approx} (Supplementary Material) 
\begin{equation*}
\|\Pi_{W_{l}}f\|_{0,p,\infty} \leq  B 2^{-l s}.
\end{equation*}
This implies with Equation~\eqref{eq:cheby} and Lemma~\ref{lem:mean} (left hand side) 
\begin{align*}
&\Pr \Big\{\exists l: j_s\leq l \leq j,  T_n(l)  \geq D^{(M)}\Big( B^p 2^{-pl s} + B^{\lfloor p\rfloor} 2^{-\lfloor p\rfloor l s} \big(\frac{2^l}{n}\big)^{(p-\lfloor p\rfloor)/2}\Big)\\ 
&+  \sqrt{\frac{C^{(p)} D^{(D)}}{\Delta}}\big( B^{p-1}2^{-ls(p-1)}\sqrt{\frac{2^{l( 1- 1/p)}}{n}} +  \sqrt{\frac{2^{(p- 1)(j+l)/2}}{n^p}}\big)  \Big\} \leq \Delta,
\end{align*}
which implies since $l\geq j_s$
\begin{align*}
&\Pr \Big\{\exists l:j_s\leq l \leq j, T_n(l)  \geq (B^p+1)\big(D^{(M)} + \sqrt{\frac{C^{(p)} D^{(D)}}{\Delta}}\big)(\tau_{p,l}(s))^p\\
&+  \sqrt{\frac{C^{(p)} D^{(D)}}{\Delta}} \sqrt{\frac{2^{(p- 1)(j+l)/2}}{n^p}}\Big\} \leq \Delta,
\end{align*}
where $(\tau_{p,l}(s))^p = 2^{-ls(p-1)}\sqrt{\frac{2^{l( 1- 1/p)}}{n}} +2^{-\lfloor p\rfloor l s} \big(\frac{2^l}{n}\big)^{(p-\lfloor p\rfloor)/2} + 2^{-pls}$.

\end{proof}

\subsection{Study of the test $\Psi_n$}

Let $1 \geq \Delta >0$ be a probability. We remind that $j$ and $j_s$ are such that 
\begin{equation*}
j_s = \lfloor \log(n^{\frac{1}{2s+1}}) \rfloor \hspace{5mm} and \hspace{5mm} j = \lfloor \log(n^{\frac{1}{2t+1 - 1/p}}) \rfloor.
\end{equation*}

We set for $j_s\leq l\leq j$, 
\begin{equation*}
\tau_l \equiv \tau_{l}(n,p,s)= C'\Big(\tau_{p,l}(s) + \sqrt{\frac{2^{(p- 1)(j+l)/(2p)}}{n}}\Big),
\end{equation*}
where
\begin{equation*}
C'\geq 2\big(D^{(M)} + \sqrt{\frac{C^{(p)}D^{(D)}}{\Delta}} (\frac{1}{D^{(m)}} + 1) + C_p\big),
\end{equation*}
and where $ \tau_{p,l}(s)$ is defined in Lemma~\ref{lem:conc0pptopgen}.

 As sums of a geometric series, the following terms entering in the composition of $(\tau_{p,l}(s))^p$ satisfy
\begin{align*}
\sum_{l=j_s}^j 2^{-pls} \equiv n^{-\frac{sp}{2s+1}}, \hspace{5mm} and \hspace{5mm}\sum_{l=j_s}^j \sqrt{\frac{2^{(p- 1)(j+l)/(2p)}}{n}} \equiv n^{-\frac{tp}{2t+1 - 1/p}},\\
and \hspace{5mm}\sum_{l=j_s}^j 2^{-ls(p-1)}\sqrt{\frac{2^{l( 1- 1/p)}}{n}} = O\Big(\max(n^{-\frac{sp}{2s+1}},n^{-\frac{tp}{2t+1 - 1/p}})\Big).
\end{align*}
Also, since $s >1/2> \frac{p - \lfloor p \rfloor}{2\lfloor p \rfloor}$, we have in the same way
\begin{equation*}
\sum_{l=j_s}^j 2^{-\lfloor p\rfloor l s} \big(\frac{2^l}{n}\big)^{(p-\lfloor p\rfloor)/2}  \equiv n^{-\frac{sp}{2s+1}}.
\end{equation*}
The last two equation blocks, together with the definition of $\tau_l$, imply that
\begin{equation*}
\sum_{l=j_s}^j \tau_l= C'A_p (n^{-\frac{s}{2s+1}} + n^{-\frac{t}{2t+1 - 1/p}}),
\end{equation*}
where $10\geq A_p\geq1$.

We set 
\begin{equation*}
t_n(l) = (B+1)\tau_l/2 \hspace{5mm} and  \hspace{5mm}\tilde t_n = (D_p + 2 C_{p/(p-1)}\sqrt{\log(1/\Delta)}) \sqrt{\frac{ 2^{j_s}}{n}},
\end{equation*}
We also write 
\begin{equation*}
\rho_n = 4\Big[(B+1)C'2^{-jt} + 2\sum_{j_s\leq l\leq j} t_n(l) + 2\tilde t_n\Big] = C(B+1) (n^{-\frac{s}{2s+1}} + n^{-\frac{t}{2t+1 - 1/p}}),
\end{equation*}
with $C$ fixed accordingly (we remind that $D^{(M)}, C^{(p)},D^{(D)}, D^{(m)}$ are strictly positive constants that depend on $p$ only).

\emph{Null Hypothesis $H_0$}

Assume that $f \in \Sigma(s,B)$. We have by Lemma~\ref{lem:conc0pptopgen} (Equation~\eqref{eq:larg4gen}) and by definition of $t_n(l)$
\begin{align*}
\Pr \Big\{\exists l: j_s\leq l \leq j,  T_n(l)   \geq (t_n(l))^{p}/2^p  \Big\} \leq \Delta.
\end{align*}

Also since $f \in \Sigma(s,B)$
\begin{align*}
\tilde T_n= \inf_{g\in \Sigma(s,B)} \|\Pi_{V_{j_s}} \hat f_n - g\|_{p}  \leq   \|\Pi_{V_{j_s}} (\hat f_n -f)\|_{p}.
\end{align*}
Thus by Lemma~\ref{lem:conc0pp}, we have
\begin{align}
\Pr \Big\{ \tilde T_n   \geq  (D_p + 2 C_{p/(p-1)}\sqrt{\log(1/\Delta)})\sqrt{  \frac{ 2^{j_s}}{n}}  \Big\} \leq \Delta,
\end{align}
and note that by definition $(D_p + 2 C_{p/(p-1)}\sqrt{\log(1/\Delta)})\sqrt{\frac{ 2^{j_s}}{n}} \leq \tilde t_n$.

So with probability $1-2 \Delta$, we have $\Psi_n=0$ under $H_0$.

\emph{Alternative hypothesis $H_1$}

We identify the seqence $(\tau_l)_{j_s \leq l \leq j}$, and the quantity $\rho_n$, with the quantities in Lemma~\ref{lem:justiftest}: they have all required properties by definition and for $C'$ (and thus $C$) large enough.

If $H_1$ is satisfied, then either $\inf_{g \in \Sigma(s,B)} \|\Pi_{V_{j_s}}( f) - g\|_{p} \geq 3\rho_n/8$ or\\
$\max_{j_s\leq l \leq j} \|\Pi_{W_{l}} (f)\|_{0,p,\infty} \geq (B+1)\tau_l$ (see Lemma~\ref{lem:justiftest}).

\emph{Case 1: $\max_{j_s\leq l \leq j} \|\Pi_{W_{l}} (f)\|_{0,p,\infty} \geq (B+1)\tau_l$}

Using the results of Lemma~\ref{lem:conc0pptopgen} (Equation~\eqref{eq:larg3gen}), we have
\begin{align*}
&\Pr \Big\{\exists l: j_s\leq l \leq j,  T_n^{(l)}   \leq (B+1)^{p-1}\tau_l^{p-1}\Big(D^{(m)}\tau_l(B+1) - \sqrt{\frac{C^{(p)} D^{(D)}}{\Delta}} \sqrt{\frac{2^{l(1-1/p)}}{n}} \Big)\\
&-  \sqrt{\frac{C^{(p)} D^{(D)}\Delta} C^{(p)}\frac{2^{(p- 1)(j+l)/2}}{n^p}}  \Big\} \leq \Delta.
\end{align*}
By defintion of $\tau_l$ (since $C'$ is large enough), we know that
\begin{equation*}
D^{(m)}\tau_l - \sqrt{C^{(p)} D^{(D)} \frac{2^{l(1-1/p)}}{n\Delta}} \geq 3\tau_l/4 .
\end{equation*}
 So by definition of $(t_n(l))^{p}$, we have 
\begin{equation*}
\frac{3}{4}(\tau_l)^{p}  -   \sqrt{\frac{C^{(p)} D^{(D)}}{\Delta} \frac{2^{(p- 1)(j+l)/2}}{n^p}} > \frac{1}{2}(\tau_l)^{p} >t_n(l).
\end{equation*}

\emph{Case 2: $\inf_{g \in \Sigma(s,B)} \|\Pi_{V_{j_s}} (f) - g\|_{p} >3\rho_n/8$ }

By triangular inequality, for any $g \in \Sigma(s,B)$
\begin{align*}
\|\Pi_{V_{j_s}} (\hat f_n) - g\|_{p} \geq \| \Pi_{V_{j_s}}(f) -g\|_{p} - \|\Pi_{V_{j_s}} (\hat f_n) - f\|_{p},
\end{align*}
which implies when combined with Lemma~\ref{lem:conc0pp}
\begin{align*}
\Pr \Big\{ \|\Pi_{V_{j_s}} (\hat f_n) - g\|_{p}   \leq  \| \Pi_{V_{j_s}}(f) -g\|_{p} - (D_p + 2 C_{p/(p-1)}\sqrt{\log(1/\Delta)})\sqrt{ \frac{2^{j_s}}{n}}  \Big\} \leq \Delta,
\end{align*}
which implies since $\inf_{g \in \Sigma(s,B)} \|\Pi_{V_{j_s}} (f )- g\|_{p} >3\rho_n/8$
\begin{align*}
\Pr \Big\{ \inf_{g \in \Sigma(s,B)} \|\Pi_{V_{j_s}} (\hat f_n) - g\|_{p}   \leq  3\rho_n/8 -  (D_p + 2 C_{p/(p-1)}\sqrt{\log(1/\Delta)})\sqrt{  \frac{ 2^{j_s}}{n}}  \Big\} \leq \Delta,
\end{align*}
so since $\tilde T_n = \inf_{g \in \Sigma(s,B)} \|\Pi_{V_{j_s}} (\hat f_n) - g\|_{p}$,
\begin{align*}
\Pr \Big\{ \tilde T_n  \leq  3\rho_n/8 - (D_p + 2 C_{p/(p-1)}\sqrt{\log(1/\Delta)})\sqrt{  \frac{ 2^{j_s}}{n}}  \Big\} \leq \Delta,
\end{align*}
and by definition we have  $3\rho_n/8 - (D_p + 2 C_{p/(p-1)}\sqrt{\log(1/\Delta)})\sqrt{  \frac{ 2^{j_s}}{n}} \geq \tilde t_n$.

So with probability $1-2\Delta$, we have $\Psi_n=1$ under $H_1$.

\emph{Conclusion on the test $\Psi_n$}

All the inequalities developed earlier are true for any $f$ in $H_0$ or $H_1$ with \textit{universal} constants (independent of $f$) and the supremum over $f$ in $H_0$ and $H_1$ of the error of type one and two are bounded by $2\Delta = \alpha/2$. Finally, the test $\Psi_n$ of error of type 1 and 2 bounded by $2\Delta = \alpha/2$ distinguishes between $H_0$ and $H_1$ with condition $\rho_n = C(B+1)  (n^{-\frac{s}{2s+1}} + n^{-\frac{t}{2t+1 - 1/p}})$ for a value $C$ large enough (but depending only on $p,\alpha$). This implies that for any $n>0$ we have
\begin{equation}\label{eq:consis}
 \sup_{f \in \Sigma(s,B)} \E_f \Psi_n +  \sup_{f \in \tilde \Sigma(t, B, \rho_n)} \E_f (1-\Psi_n) \leq \alpha.
\end{equation}
The test is $\alpha-$consistent (see~\citep{ingster2002nonparametric} for a definition).

\section{Proof of Theorem~\ref{th:sbig} (upper bound)}\label{s:sbigcbcb}

Let $1\geq \Delta >0$. We know that for $n$ and $C$ (in the definition of $\rho_n$) large enough (depending only on $p, B,\Delta$), there exists a test $\Psi_n$ for the testing problem~\eqref{eq:testhaha} that is consistent and with level $\Delta$ (see Theorem~\ref{th:test}).

Set $U_p := U_p(B)$ where $U_p(B)$ is the constant in Theorem~\ref{th:inference}. Consider the confidence set around the adaptive estimate $\hat f_n(\hat j_n)$ (where $\hat f_n(\hat j_n)$ is constructed as in Theorem~\ref{th:inference}) as being 
\begin{equation*}
C_n = \Big\{g: \|\hat f_n(\hat j_n) - g\|_p \leq \frac{1}{\Delta}U_p n^{-\frac{s}{2s+1}} \ind{\Psi_n=0} + \frac{1}{\Delta}U_p n^{-\frac{t}{2t+1}} \ind{\Psi_n=1} \Big\}.
\end{equation*}

Then since the test $\Psi_n$ is $\Delta-$consistent
\begin{equation*}
\sup_{f \in \Sigma(s,B)} {\Pr}_{f}\Big(|C_n| > \frac{1}{\Delta}U_p n^{-\frac{s}{2s+1}}\Big)  = \sup_{f \in \Sigma(s,B)} \E_f \Psi_n \leq \Delta,
\end{equation*}
and
\begin{equation*}
\sup_{f \in \tilde \Sigma(t, B, \rho_n)} {\Pr}_{f}\Big(|C_n| > \frac{1}{\Delta}U_p n^{-\frac{t}{2t+1}}\Big)  = 0.
\end{equation*}

Also we have by Markov's inequality
\begin{align*}
\sup_{f \in \Sigma(s,B)} {\Pr}_{f}\Big( f \in C_n \Big) &\geq 1-\sup_{f \in \Sigma(s,B)} {\Pr}_{f}\Big( \|\hat f_n(\hat j_n) - f\|_p \geq \frac{1}{\Delta}U_p n^{-\frac{s}{2s+1}}\Big)\\
&\geq 1-\frac{\Delta n^{\frac{s}{2s+1}}}{U_p} \sup_{f \in \Sigma(s,B)} \E\|\hat f_n(\hat j_n) - f\|_p\\
&\geq 1-\frac{\Delta n^{\frac{s}{2s+1}}}{U_p}  U_p n^{-\frac{s}{2s+1}}  \\
&\geq 1- \Delta,
\end{align*}
where we use Theorem~\ref{th:inference} for the bound on $\E\|\hat f_n(\hat j_n) - f\|_p$, and we also have still by Markov's inequality
\begin{align*}
\sup_{f \in \tilde \Sigma(t, B, \rho_n)} {\Pr}_{f}\Big( f \in C_n \Big) &\geq 1-\sup_{f \in \tilde \Sigma(t, B, \rho_n)} {\Pr}_{f}\Big( \|\hat f_n(\hat j_n) - f\|_p \geq \frac{1}{\Delta}U_p n^{-\frac{t}{2t+1}}\Big)\\
&- \sup_{f \in \tilde \Sigma(t,B, \rho_n)} \E_f(1 - \Psi_n)\\
&\geq 1-\frac{\Delta n^{\frac{t}{2t+1}}}{U_p} \sup_{f \in \tilde \Sigma(t, B,\rho_n)} \E\|\hat f_n(\hat j_n) - f\|_p\\ 
&- \sup_{f \in \tilde \Sigma(t, B,\rho_n)} \E_f(1 - \Psi_n)\\
&\geq 1-\frac{\Delta n^{\frac{t}{2t+1}}}{U_p}  U_p n^{-\frac{t}{2t+1}}  - \Delta \\
&\geq 1- 2\Delta.
\end{align*}

These four inequalities imply that $C_n$ is an $(L_p, 2\Delta)-$honest and adaptive confidence bound on $\Sigma(s,B) \cup \tilde \Sigma(t, B, \rho_n)$ for $\{s,t\}$ and $B$.

\section{Proof of Theorem~\ref{th:ssmallseg1}}\label{sec:ssmallseg1}

Let $\Psi_n$ be a test as defined in Theorem~\ref{th:test}, and 
\begin{equation*}
\rho_n =C(B+1) \max(n^{-t/(2t+1-1/p)}, n^{-s/(2s+1)}) = C(B+1) n^{-s/(2s+1)},
\end{equation*}
where $C$ defined as in Theorem~\ref{th:test}. The confidence set we consider is the following.
\begin{align*}
C_n = \Big\{g\in \Sigma(t,B): \|\tilde f_n - g\|_p \leq \frac{D}{\alpha} n^{-\frac{s}{2s+1}} (1-\Psi_n) + \frac{D}{\alpha} n^{-\frac{t}{2t+1}} \Psi_n\Big\},
\end{align*}
where $\tilde f_n$ is the adaptive estimate of Theorem~\ref{th:inferenceersatz} (which is actually the same estimate than the one for Theorem~\ref{th:inference}) and where $D\geq  \tilde U_p(B,2C)$ where $\tilde U_p(B,2C)$ is the constant defined in Theorem~\ref{th:inferenceersatz}.

By Markov's inequality
\begin{align*}
\sup_{f \in \Sigma(t,B)}\mathbb P(f \not \in C_n) &\leq  \sup_{f \in \tilde \Sigma(t,B, \rho_n)}\mathbb P_f(f \not \in C_n) + \sup_{f \in \Sigma(t,B)\setminus \tilde \Sigma(t,B, \rho_n)}\mathbb P_f(f \not \in C_n)\\
&\leq \sup_{f \in \tilde \Sigma(t,B, \rho_n)}\mathbb P_f(\|\tilde f_n - f\|_p \geq \frac{D}{\alpha} n^{-\frac{t}{2t+1}})+ \sup_{f \in \tilde \Sigma(t,B, \rho_n)}\mathbb E_f[1-\Psi_n]\\
&+ \sup_{f \in \Sigma(t,B): \|f -  \Sigma(s,B)\|_p \leq C(B+1) n^{-s/(2s+1)}}\mathbb P_f(\|\tilde f_n - f\|_p \geq \frac{D}{\alpha} n^{-\frac{s}{2s+1}})\\
&\leq \sup_{f \in \tilde \Sigma(t,B, \rho_n)}\mathbb E_f(\|\tilde f_n - f\|_p) \frac{\alpha}{D} n^{\frac{t}{2t+1}}+ \alpha\\
&+ \sup_{f \in \Sigma(t,B): \|f -  \Sigma(s,B)\|_p \leq C(B+1) n^{-s/(2s+1)}}\mathbb E_f(\|\tilde f_n - f\|_p)  \frac{\alpha}{D} n^{\frac{s}{2s+1}})\\
&\leq 3\alpha,
\end{align*}
by Theorem~\ref{th:test} (since $\rho_n= C(B+1) n^{-s/(2s+1)}$) and by Theorem~\ref{th:inferenceersatz} (since $D\geq  \tilde U_p(B,2C)$). Also it holds that
\begin{align*}
\sup_{f \in  \Sigma(s,B)}\mathbb P\Big[ |C_n| \geq \frac{D}{\alpha}n^{-\frac{s}{2s+1}} \Big] \leq \alpha,
\end{align*}
by definition of $\Psi_n$ and by Theorem~\ref{th:test}, and
\begin{align*}
|C_n| \leq \frac{D}{\alpha}n^{-\frac{t}{2t+1}}.
\end{align*}
This concludes the proof.

\section{Proof of Theorem~\ref{th:test} (lower bound)}\label{sec:lowbound}

Let $1> \upsilon>0$, and $j \in \mathbb N^*$ such that $2^j \approx n^{\frac{1}{2t+1-1/p}}$. 

\emph{Step 1: Definition of a testing problem.}

We define the following prior $\tilde \Pi$ on for a sequence $(\alpha_{l,k})_{l \geq J_0, k \in Z_l}$:
\begin{equation*}
\alpha \sim \tilde \Pi \Leftrightarrow \forall l \neq j,  \forall k \in Z_l, \alpha_{l,k}=0,  \forall k \in Z_j, \alpha_{l,k} = B_k,
\end{equation*}
where the $B_k$ are i.i.d.Bernoulli of parameter $2^{-j/2}$.

Consider the sequence of coefficients indexed by a given $\alpha \in I$ as
\begin{equation*}
 a_{l,k}^{(\alpha)} = \upsilon a \alpha_{l,k},
\end{equation*}
where $a = 2^{-j(t + 1/2 - 1/(2p))} = n^{-1/2}$. Consider $f^{(\alpha)}$ the function associated to $a^{(\alpha)}$, i.e.
\begin{equation*}
 f^{(\alpha)} = \sum_{l\geq J_0} \sum_{k \in Z_l} a_{l,k}^{(\alpha)}.
\end{equation*}
We write by a slight abuse of notations that $f \sim \tilde \Pi$ if $f = f^{(\alpha)}$ where $\alpha \sim \tilde \Pi$.

Consider the testing problem
\begin{equation}\label{testprime}
H_0: f=0 \hspace{5mm}vs. \hspace{5mm} H_1: f \sim \tilde \Pi.
\end{equation}

\emph{Step 2: Quantity of interest.}

Let $\Psi$ be a test, that is to say a measurable function that takes values in $\{0,1\}$. Equivalently to having access to the process $Y^{(n)}$, we have access to the coefficients $(\hat a_{l,k})_{l,k}$ and each of these coefficients are independent $\mathcal N(a_{l,k},1/n)$.

We have for any $\eta>0$
\begin{align}
\E_0[\Psi] + \E_{\alpha \sim \tilde\Pi}\E_{f^{(\alpha)}}[1-\Psi] &\geq \E_0[\Psi] + \E_{\alpha \sim \tilde\Pi}\E_{f^{(\alpha)}}[1-\Psi] \nonumber\\
&\geq \E_0\Big[ \ind{\Psi=1}] + \ind{\Psi=0} Z \Big] \nonumber\\
&\geq (1-\eta) \Prob_0(Z \geq 1-\eta) \label{eq:boundZ},
\end{align}
where $Z= \E_{\alpha \sim \tilde \Pi} \prod_{l,k} \frac{dP_{l,k}^{(\alpha)}}{dP_{l,k}^{(0)}}$, where $dP_{l,k}^{(\alpha)}$ is the distribution of $\hat a_{l,k}$ when the function generating the data is $f^{(\alpha)}$, and $dP_{l,k}^{(0)}$ is the distribution of $\hat a_{l,k}$ when the function generating the data is $0$ (this holds since the $(\hat a_{l,k})_{l,k}$ are independent).

More precisely, we have since the $(\hat a_{l,k})_{l,k}$ are independent $\mathcal N(a_{l,k},1/n)$
\begin{align*}
Z((x_{k})_{k}) &=\E_{\alpha \sim \tilde \Pi} \Bigg[ \prod_{l> 0, k \in Z_l} \frac{\exp(-\frac{n}{2}(x_{l,k} - a_{l,k}^{(\alpha)})^2)}{\exp(-\frac{n}{2}x_{l,k}^2)} \Bigg]\\
&=\E_{\alpha \sim \tilde\Pi}\Bigg[ \prod_{k\in Z_j} \exp(n x_{k}  a_{k}^{(\alpha)}) \exp(- \frac{n}{2} (a_{k}^{(\alpha)})^2) \Bigg],
\end{align*}
where we simplify notations by setting $x_k \equiv x_{j,k}$ and $a_k^{(\alpha)} \equiv a_{j,k}^{(\alpha)}$. We also write later $\alpha_k \equiv \alpha_{j,k}$.

By Markov and Cauchy Schwarz's inequality
\begin{align}\label{eq:mark}
\Prob_0(Z \geq 1-\eta) \geq 1 - \frac{\E_0|Z-1|}{\eta} \geq 1 - \frac{\sqrt{\E_0(Z-1)^2}}{\eta}.
\end{align}



We have by definition of $Z$
\begin{small}
\begin{align}
&\E_0 \big[(Z - 1)^2 \big] \nonumber\\ 
&= \int_{x_1,...x_{2^j}} \Bigg(\E_{\alpha \sim \tilde \Pi} \Big[\prod_k \exp(x_{k}  n a_{k}^{(\alpha)}) \exp(- \frac{n}{2} (a_{k}^{(\alpha)})^2)\Big] - 1\Bigg)^2 \prod_k \frac{1}{\sqrt{2n\pi}}\exp(- \frac{n}{2} (x_{k})^2) dx_1...x_{2^j} \nonumber\\
&= \E_{\alpha, \alpha' \sim \tilde \Pi} \Bigg[ \prod_k \int_{x_k} \exp(x_{k}  n (a_{k}^{(\alpha)} + a_{k}^{(\alpha')})) \exp(- \frac{n}{2}((a_{k}^{(\alpha)})^2 + (a_{k}^{(\alpha')})^2)) \frac{1}{\sqrt{2n\pi}}\exp(- \frac{n}{2} (x_{k})^2)dx_k \Bigg] - 1 \nonumber\\
&= \E_{\alpha, \alpha' \sim \tilde \Pi}   \Bigg[ \prod_k \Big((1-\ind{\alpha_k=\alpha_k' = 1}) + \exp(n \upsilon^2 a^2) \ind{\alpha_k = \alpha_k' = 1} \Big) \Bigg] - 1 \nonumber\\
&=  \E_{\alpha, \alpha' \sim \tilde \Pi} \Bigg[ \prod_k \Big(1 + (\exp(n \upsilon^2 a^2) -1)\ind{\alpha_k = \alpha_k' = 1} \Big) \Bigg]  - 1. \label{eq:Z}
\end{align}
\end{small}
Since all $(\alpha_k)_k, (\alpha_k')_k$ are i.i.d.~Bernoulli of parameter $2^{-j/2}$, it implies that the $(\ind{\alpha_k = \alpha_k' = 1})_k$ are i.i.d.~Bernoulli random variables of parameter $2^{-j}$. This implies together with Equation~\eqref{eq:Z} that
\begin{align}
\E_0 \big[(Z - 1)^2 \big] &=  \prod_k  \E_{B_k' \sim \mathcal B(2^{-j})} \Big(1 + (\exp(n \upsilon^2 a^2) -1) B_k' \Big) \Bigg]  - 1 \nonumber\\
&=   \Big(1 + (\exp(n \upsilon^2 a^2) -1) \frac{1}{2^j} \Big)^{2^j}  - 1 \nonumber\\
&\leq   \exp((\exp(n \upsilon^2 a^2) -1))   - 1. \nonumber
\end{align}
where $\mathcal B(2^{-j})$ is the law of a Bernoulli of parameter $2^{-j}$, and since for any $u\geq 0$, $1+u \leq \exp(u)$. Since $a^2 = n^{-1}$, we get
\begin{align}
\E_0 \big[(Z - 1)^2 \big] &\leq   \exp((\exp( \upsilon^2) -1)) - 1 \nonumber\\
&\leq   \exp(2 \upsilon^2) - 1 \leq   4 \upsilon^2, \label{eq:lescococ}
\end{align}
since for $u \leq 1$, we have $\exp(u) \leq 1 + 2u$.

\emph{Step 3: Conclusion on the test~\ref{testprime}}

By combining Equations~\eqref{eq:boundZ}, ~\eqref{eq:mark},  and~\ref{eq:lescococ} we know that
\begin{align*}
\E_0[\Psi] + \E_{\alpha \sim \tilde\Pi} \E_{f^{(\alpha)}}[1-\Psi] \geq 1 - 4 \upsilon^2,
\end{align*}
and since this holds with any $\Psi$, we have
\begin{align}\label{eq:noconsi}
\inf_{\Psi}\Big[ \E_0[\Psi] + \E_{\alpha \sim \tilde\Pi} \E_{f^{(\alpha)}}[1-\Psi] \Big] \geq 1 - 4 \upsilon^2,
\end{align}
and this implies that there is no $1 - 4 \upsilon^2-$consistent test for the testing problem~\eqref{testprime}.

\emph{Step 4: Extension of this result to a deterministic testing problem.}

Define the set
\begin{equation*}
I = \Big\{(\alpha_{l,k})_{l,k}: \forall l \neq j, \alpha_{l,k} =0, \alpha_{j,k}\in\{0,1\}, \sum_{k\in Z_j} \alpha_{j,k} = S, \frac{2^{j/2}}{2} \leq S \leq \frac{3}{2} 2^{j/2} \Big\}.
\end{equation*}
Consider the associated sequence of coefficients indexed by $\alpha \in I$, and the corresponding function $f^{(\alpha)}$. Consider the testing problem
\begin{equation}\label{testprimeprime}
H_0: f=0 \hspace{5mm}vs. \hspace{5mm} H_1: f =f^{(\alpha)}, \alpha \in I.
\end{equation}

Consider now $\alpha \sim \tilde \Pi$. By Hoeffding's inequality, we know that for $\lambda \leq 2^{j/2}$, we have
\begin{align*}
{\Pr}_{\alpha \sim \tilde \Pi} (|\sum_k \alpha_{j,k} - 2^{j/2}|\geq   \lambda) \leq 2 \exp(- 2^{-j/2}\lambda^2/2).
\end{align*}
Let $\lambda = \frac{2^{j/2}}{2}$. Then the last equation implies
\begin{align*}
{\Pr}_{\alpha \sim \tilde \Pi} (|\sum_k \alpha_{j,k} - 2^{j/2}|\geq  \frac{2^{j/2}}{2}) \leq 2 \exp(-2^{j/2}/8),
\end{align*}
so this implies in particular that with $\tilde \Pi-$probability larger than $1 - 2 \exp(-2^{j/2}/8)$, we have $\alpha \in I$.

For any test $\Psi$, since $\Psi \leq 1$, we have 
\begin{equation*}
 \E_{\alpha \sim \tilde\Pi} \E_{f^{(\alpha)}}[1-\Psi]  \leq {\Pr}_{\alpha \sim \tilde\Pi}( \alpha \not\in I) + \sup_{\alpha \in I}  \E_{f^{(\alpha)}}[1-\Psi] .
\end{equation*}
This implies since with $\tilde \Pi-$probability larger than $1 - 2 \exp(-2^{j/2}/8)$, $\alpha \in I$, that
\begin{equation*}
 \E_{\alpha \sim \tilde\Pi} \E_{f^{(\alpha)}}[1-\Psi]  \leq \sup_{\alpha \in I}  \E_{f^{(\alpha)}}[1-\Psi]  + 2 \exp(-2^{j/2}/8).
\end{equation*}
This implies when combined to Equation~\eqref{eq:noconsi} that
\begin{align*}
\inf_{\Psi}\Big[ \E_0[\Psi] + \sup_{\alpha \in I}  \E_{f^{(\alpha)}}[1-\Psi]  \Big] &\geq 1 - 4 \upsilon^2 - 2 \exp(-2^{j/2}/8)\\
&\leq 1 -5 \upsilon^2,
\end{align*}
for $n$ and thus $j$ large enough. This implies that there is no $1 - 5 \upsilon^2-$consistent test for the testing problem~\eqref{testprimeprime}, and this holds for any $\upsilon >0$.

\section{Proof of Theorem~\ref{th:sbig} (lower bound)}\label{th:lbcs}

Consider all the quantities defined in Section~\ref{sec:lowbound}.

Assume that $s(1-1/p)>t$. Let $\upsilon\leq B/2$. Set 
\begin{equation*}
\rho_n = C_p\frac{\upsilon  n^{-\frac{t}{2t+1 - 1/p}}}{4}.
\end{equation*}
Let $\alpha \in I$.

By triangular inequality, 
\begin{equation*}
\|f^{(\alpha)}\|_{t,p, \infty}  = \big(\sum_k (a_{j,k}^{(\alpha)})^p2^{jp(1/2-1/p)}\big)^{1/p} =\frac{3}{2} \upsilon \leq B,
\end{equation*}
which implies that $f^{(\alpha)} \in \Sigma(t,B)$. Also that since only the $j-th$ first coefficients of $f$ are non-zero, and since $\|.\|_p \geq C_p \|.\|_{0,p,p}$ (see Proposition~\ref{prop:embeed} in the supplementary Material), we have 
\begin{align*}
\|f^{(\alpha)} - \Sigma(s,B)\|_{p} &\geq C_p \|f^{(\alpha)} - \Sigma(s,B)\|_{0,p,p}\\
&= C_p \min_{g \in \Sigma(s,B)} \|\Pi_{W_{j}} (f^{(\alpha)}) - \Pi_{W_{j}} (g)\|_{0,p,p}\\
&\geq C_p \|\Pi_{W_{j}} f^{(\alpha)}\|_{0,p,p}/2 - C_p B 2^{-j s}\\
&\geq C_p\frac{\upsilon n^{-\frac{t}{2t+1 - 1/p}}}{4}.
\end{align*}
by triangular inequality since for any $g \in \Sigma(s,B), \|g\|_{0,p,p} \leq B 2^{-j s}\leq \upsilon n^{-\frac{t}{2t+1 - 1/p}}/4$ for $n$ large enough, since $s(1-1/p)>t$. This implies in particular that  $f^{(\alpha)}\in \tilde \Sigma(t,B,\rho_n)$

To sum up, 
\begin{equation*}
0 \in \Sigma(s,B) \hspace{5mm} and \hspace{5mm} \forall \alpha, f^{(\alpha)}\in \tilde \Sigma(t,B,\rho_n).
\end{equation*}

Assume that there exists some honest and adaptive confidence set $C_n$ for $\mathcal P_n = \Sigma(s,B) \cup  \Sigma(t,B,\rho_n)$, and $\{s,t\}$.

This implies that the confidence set $C_n$ is in particular honest and adaptive over $\mathcal P_n' = \{0\} \cap \{f^{(\alpha)}, \alpha \in I\}$. So, for any $0<\delta<1$, there exists a constant $L$ (that might depend on $B,s,t,\delta$) such that for $n>0$
\begin{equation*}
 {\Pr}_{0}\Big(\{|C_n| \geq L n^{-s/(2s+1)}\} \cup \{0 \not\in C_n\}  \Big) \leq 2\delta
\hspace{3mm}and \hspace{3mm}
 \sup_{\alpha \in I}{\Pr}_{f ^{(\alpha)}}\Big( f ^{(\alpha)} \in C_n \Big) \geq 1-\delta.
\end{equation*}

We define a test $\Psi$ as follows. If $0 \in C_n$ and $|C_n| \leq L n^{-s/(2s+1)}$, then $\Psi = 0$, otherwise $\Psi = 1$. We have for $n>0$
\begin{equation*}
{\Pr}_0 (\Psi = 1)  =  {\Pr}_{0}\Big(\{|C_n| \geq L n^{-s/(2s+1)}\} \cup \{0 \not\in C_n\}  \Big) \leq 2\delta.
\end{equation*}
Also
\begin{equation*}
 \sup_{\alpha \in I} {\Pr}_{f ^{(\alpha)}}\Big( \Psi = 0) \leq 
 \sup_{\alpha \in I} {\Pr}_{f ^{(\alpha)}}\Big( \|f ^{(\alpha)} - 0\|_p \leq  L n^{-s/(2s+1)}\Big) \leq  \sup_{\alpha \in I}{\Pr}_{f ^{(\alpha)}}\Big( f ^{(\alpha)} \not\in C_n \Big) \geq \delta.
\end{equation*}
Combining both results imply that for $n$ large enough, there exists a consistent test $\Psi$ constructed using $C_n$, that is to say such that
\begin{align*}
\inf_{\Psi}\E_0[\Psi] + \sup_{f^{(\alpha)}, \alpha \in I}\E_{f^{(\alpha)}}[1-\Psi] \leq 3 \delta,
\end{align*}
and that for any $\delta>0$. This is in contradiction with the result of Step 4 (no consistent test for the testing problem~\ref{testprimeprime}), and we deduce by contradiction that no honest and adaptive confidence set exists on $\mathcal P_n$. This concludes the proof.


\section*{Acknowledgments.} I would like to thank Richard Nickl for enlightening and insightful discussions, as well as careful re-reading and comments. I also would like to thank the reviewers and editors for many helpful comments.

\section*{Supplementary material}

\section{Technical preliminary results}\label{sec:prelresults}

In this Section, we remind some well-known preliminary results, which we sometimes extend, or adapt.

We first provide the following Assumption.

\begin{assumption}\label{ass:wav2}
We assume that there is a universal constant $C_p$ such that for any $(l,k)$ we have
\begin{equation*}
\|\psi_{l,k}\|_p \leq C_p 2^{l/2},
\end{equation*}
and
\begin{equation*}
\|\phi_{k}\|_p \leq C_p.
\end{equation*}
\end{assumption}
Note that Assumption~\ref{ass:wav1} implies Assumption~\ref{ass:wav2}.

\subsection{Properties of Besov spaces}

We remind the following Proposition (see~\citep{bergh1976interpolation} or~\citep{besov1978integral}, volume 2, Chapter 18, page 68)
\begin{proposition}\label{prop:embeed}
Assume that $p\geq 2$ (and $p <\infty$). Then
\begin{equation*}
B_{0,p,2} \subset L_p \subset B_{0,p,p}.
\end{equation*}
If $p'\leq 2$, then we have
\begin{equation*}
B_{0,p',p'} \subset L_{p'} \subset B_{0,p',2}.
\end{equation*}
\end{proposition}

We also remind the following Proposition (see also~\citep{hardle1998wavelets}).
\begin{proposition}\label{prop:approx}
Let $s>0$, $p\geq 2$ (and $p <\infty$) and $h \geq 2$. Assume that $f \in B_{s,p,\infty}$. Then
\begin{equation*}
\|f - \Pi_{V_j}f\|_{0,p,h} \leq \|f\|_{s,p,\infty} 2^{-js},
\end{equation*}
and also
\begin{equation*}
\|f - \Pi_{V_j}f\|_{p} \leq C_p \|f\|_{s,p,\infty} 2^{-js}.
\end{equation*}
Note that this is also satisfied for the weaker condition $p'\geq 1$ and $h \geq 1$ (by just remarking that $L_{p'} \subset B_{0,p',1}$ also for any $p'\geq 1$).
\end{proposition}
\begin{proof}
We have
\begin{align*}
\|f - \Pi_{V_j}f\|_{0,p,h} &= \left( \sum_{l\geq j} 2^{lh(1/2-1/p)}|<f,\psi_{l,.}>|_{l_p}^h \right)^{1/h} \\
&\leq 2^{-js} \left( \sum_{l\geq j} 2^{lh(s+1/2-1/p)}|<f,\psi_{l,.}>|_{l_p}^h \right)^{1/h} \\
&\leq \|f\|_{s,p,\infty} 2^{-js}.
\end{align*}

For the $\|.\|_p$ norm, it comes from the fact that there exists a constant $C_p$ such that $\|f\|_p \leq C_p\|f\|_{0,p,2}$ for any $f$ (Proposition~\ref{prop:embeed}).
\end{proof}

\subsection{Behaviour of thresholded wavelet estimates}

We also remind Rosenthal's inequality (see~\citep{hardle1998wavelets}, page 132)
\begin{proposition}\label{prop:rosenthal}
 Let $(X_1, \ldots, X_n)$ be $n$ i.i.d.~random variables such that $\E X_i = 0$ and for a given $p \geq 2$ (and $p <\infty$), $\E|X_i|^p <\infty$. Then there exists a universal constant $\tilde C_p^p$ such that
\begin{equation*}
 \E \Big|\sum_{i=1}^n X_i \Big|^p \leq \tilde C_p^p \Big(\sum_{i=1}^n \E|X_i|^p + \big(\sum_{i=1}^n \E X_i^2\big)^{p/2} \Big)
\end{equation*}
\end{proposition}

We remind the following Proposition (see~\citep{gine2012rates}, and here we provide an alternative proof).

\begin{proposition}\label{prop:est}
If Assumption~\ref{ass:wav2} is satisfied, there exists a universal constant $D_p$ that depends on $p$ only such that for any fixed $j \in \mathbb N^*$ we have
\begin{equation*}
\E\|\Pi_{V_j} \hat f_n -\Pi_{V_j} f\|_p^p \leq D_p^p \frac{2^{jp/2}}{n^{p/2}}.
\end{equation*}
\end{proposition}
\begin{proof}
Let $x \in [0,1]$. We have
\begin{equation*}
\E \Big|\Pi_{V_j} \hat f_n(x) -\Pi_{V_j} f(x)\Big|^p = \E \Big|\sum_k \frac{1}{\sqrt{n}}G'_k \phi_k(x) + \sum_{l\leq j,k} \frac{1}{\sqrt{n}}G_{l,k} \psi_{l,k}(x)\Big|^p,
\end{equation*}
where $G_{l,k} = \sqrt{n} (\hat a_{l,k} - a_{l,k})$ and $G'_k = \sqrt{n} (\hat a_k' - a_k')$ and the $(G_{k'}',G_{l,k})_{k',l,k}$ are thus i.i.d. gaussian random variables of mean $0$ and variance $1$. In order to simplify the notations, we abuse notations and set $\psi_{-1,k} = \phi_k$ and $G_{-1,k} = G_k'$.

We use Rosenthal's inequality (Proposition~\ref{prop:rosenthal}), and obtain
\begin{align*}
\E \Big|\Pi_{V_j} \hat f_n(x) -\Pi_{V_j} f(x)\Big|^p &= \E \Big|\sum_{-1\leq l\leq j,k} \frac{1}{\sqrt{n}}G_{l,k} \psi_{l,k}(x)\Big|^p\\
&\leq \tilde C_p^p\Bigg(\sum_{-1\leq l\leq j,k} \frac{1}{n^{p/2}}\E|G_{l,k}|^p |\psi_{l,k}(x)|^p\\ 
&+ \Big(\sum_{-1\leq l\leq j,k} \frac{1}{n}\E G_{l,k}^2 \psi_{l,k}(x)^2\Big)^{p/2}\Bigg)\\
&= \tilde  C_p^p\Bigg(\sum_{-1\leq l\leq j,k} \frac{1}{n^{p/2}}c_p^p |\psi_{l,k}(x)|^p + \Big(\frac{1}{n} \sum_{-1\leq l\leq j,k} \psi_{l,k}(x)^2\Big)^{p/2}\Bigg)\\
&\leq \tilde  C_p^p\Bigg(\sum_{-1\leq l\leq j,k} \frac{1}{n^{p/2}}c_p^p |\psi_{l,k}(x)|^p + \frac{C_p^p2^{jp/2}}{n^{p/2}}\Bigg).
\end{align*}
where $c_p^p$ is the $p$-th moment of a $\mathcal N(0,1)$, and since Assumption~\ref{ass:wav2} is satisfied.

By Assumption~\ref{ass:wav2}, we know that for any $(l,k)$ we have
\begin{equation*}
 \int |\psi_{l,k}(x)|^{p} dx \leq C_p^p 2^{lp(1/2 - 1/p)},
\end{equation*}
and thus we have that
\begin{align*}
\int \E \Big|\Pi_{V_j} \hat f_n(x) -\Pi_{V_j} f(x)\Big|^p dx  &\leq \tilde  \int_x C_p^p\Bigg(\sum_{-1\leq l\leq j,k} \frac{1}{n^{p/2}}c_p^p |\psi_{l,k}(x)|^p + \frac{C_p^p2^{jp/2}}{n^{p/2}}\Bigg)dx\\
&\leq \tilde  C_p^p\Bigg(\sum_{-1\leq l\leq j,k} \frac{1}{n^{p/2}}c_p^p C_p^p 2^{lp(1/2 - 1/p)}+ \frac{C_p^p2^{jp/2}}{n^{p/2}}\Bigg)\\
&\leq \tilde  C_p^p(c_p^p C_p^p  +C_p^p) \frac{2^{jp/2}}{n^{p/2}}.
\end{align*}

This concludes the proof.

\end{proof}

We finally state the following Proposition (it is a generalisation of what is done in~\citep{gine2012rates}).
\begin{proposition}\label{prop:Borell}
 Let $2\leq p\leq \infty$ and $2\leq h\leq \infty$. Then
\begin{align*}
\Pr \Big\{ \big| \|\Pi_{V_{j}}(\hat f_n - f)\|_{0,p,h}   - \E [ \|\Pi_{V_{j}}(\hat f_n - f)\|_{0,p,h}] \big| \geq 2\sqrt{\log(1/\delta) \frac{ 2^{j(1- 2/p)}}{n}}  \Big\} \leq \delta,
\end{align*}
and in particular, this implies
\begin{align*}
\Pr \Big\{ \big| \|\Pi_{V_{j}}(\hat f_n - f)\|_{p}   - \E [ \|\Pi_{V_{j}}(\hat f_n - f)\|_{p}] \big| \geq 2C_{p/(p-1)}\sqrt{\log(1/\delta) \frac{ 2^{j(1- 2/p)}}{n}}  \Big\} \leq \delta,
\end{align*}
\end{proposition}
\begin{proof}
We first remind Borell's inequality:
\begin{theorem}[Borell's inequality]
Let $(G(t))_{t\in T}$ be a centred Gaussian process indexed by a countable set $T$ such that $\sup_{t \in T} G(t) < +\infty$ almost surely. Then $\E \sup_{t \in T} G(t) < +\infty$ and for every $r \geq 0$, we have
\begin{equation*}
 \Pr \Big(\big|\sup_{t \in T} G(t) - \E \sup_{t \in T} G(t)\big|\geq r \Big) \leq 2 \exp(-r^2/2\sigma^2),
\end{equation*}
where $\sigma^2 = \sup_{t\in T} \E G^2(t) < +\infty$.
\end{theorem}
We use the separability of the ball of radius $1$ of $B_{0,p/(p-1),h/(h-1)}$ (that we write $B_0$ to prove that by Borell's inequality (since $(\Pi_{V_j} (\hat f_n -f) = \sum_{l\leq j,k} (\hat a_{l,k} - a_{l,k}) \psi_{l,k}(x))_{x\in [0,1]}$ is a centered Gaussian process, and thus $(<\Pi_{V_j} (\hat f_n -f), g>)_{g\in B_0}$ is a centred Gaussian process):
\begin{align*}
\Pr \Big\{ \big| \sup_{g \in B_0}<\Pi_{V_j} (\hat f_n -f), g>   - \E [ \sup_{h \in B_0}<\Pi_{V_j} (\hat f_n -f), g> \big| \geq 2\sqrt{\log(1/\delta)} \sigma \Big\} \leq \delta,
\end{align*}
where $\sigma^2 \leq \sup_{g \in B_0}\E [ <\Pi_{V_{j}}(\hat f_n - f), g>^2 ]$.

Note first that by Hahn-Banach's duality Theorem (since $\|.\|_{0,p,h}$ is the norm corresponding to $B_0$), we have that
\begin{equation*}
 \sup_{g \in B_0}<\Pi_{V_j} (\hat f_n -f), g> = \|\Pi_{V_j} (\hat f_n -f)\|_{0,p,h},
\end{equation*}
and we can rewrite the previous equation as
\begin{align*}
\Pr \Big\{ \big| \|\Pi_{V_j} (\hat f_n -f)\|_{0,p,h}   - \E [ \|\Pi_{V_{j}}(\hat f_n - f)\|_{0,p,h}] \big| \geq 2\sqrt{\log(1/\delta)} \sigma \Big\} \leq \delta.
\end{align*}

Concerning $\sigma^2$, we have (since the $(\hat a_{l,k} - a_{l,k})$ are independent centered Gaussian of variance $1/n$)
\begin{align*}
 \sup_{u \in B_0}\E [ <\Pi_{V_{j}}(\hat f_n - f), u>^2] &= \sup_{u \in B_0}\E [ <\sum_{l\leq j,k} (\hat a_{l,k} - a_{l,k}) \psi_{l,k}, u>^2]\\
&= \sup_{u \in B_0}\E \Big[ \big(\sum_{l\leq j,k} (\hat a_{l,k} - a_{l,k}) <\psi_{l,k}, u>\big)^2\Big]\\
&= \sup_{u \in B_0} \sum_{l\leq j,k}\frac{1}{n} \big(<\psi_{l,k}, u>\big)^2\\
&= \sup_{u \in B_0} \frac{1}{n} \|\Pi_{V_j} u\|_2^2.
\end{align*}

We are thus interested in computing $\sup_{u \in B_0} \|\Pi_{V_j} u\|_2^2$, i.e.~the maximum squared $L_2$ norm of a vector of $\|.\|_{0,p/(p-1),h/(h-1)}$ norm of $1$. We have by Plancherel's theorem
\begin{align*}
\sup_{u \in B_0} \|\Pi_{V_j} u\|_2^2 &= \sup_{u \in B_0} \sum_{l\leq j,k} u_{l,k}^2.
\end{align*}
Let us consider $u \in B_0$. We have, since $p \geq 2$ (and $p <\infty$)
\begin{align*}
\sum_{k} u_{l,k}^2 &\leq |u_{l,.}|_{p/(p-1)}^2.
\end{align*}
Also, since $h\geq 2$
\begin{align*}
\sum_{l\leq j} |u_{l,.}|_{p/(p-1)}^2  &= \sum_{l\leq j} |u_{l,.}|_{p/(p-1)}^2 2^{2l(1/2 - (p-1)/p)} 2^{-2l(1/2 - (p-1)/p)}\\
&\leq \Big(\sum_{l\leq j} |u_{l,.}|_{p/(p-1)}^{\frac{h}{h-1}} 2^{l(1/2 - (p-1)/p)\frac{h}{h-1}}\big)^{(h-1)/h} \Big)^2 2^{2j(1/2 - 1/p)}\\
&\leq \|\Pi_{V_j}u\|_{0,p/(p-1), h/(h-1)}^2 2^{j(1 - 2/p)}.
\end{align*}
When putting all this together, we obtain finally
\begin{align*}
\sigma^2 \leq \frac{1}{n}\sup_{u \in B_0} \|\Pi_{V_j} u\|_2^2 \leq \frac{2^{j(1 - 2/p)}}{n},
\end{align*} 
which concludes the proof for the $\|.\|_{0,p,h}$ norm.

For the $\|.\|_p$ norm, we apply as before Borell's inequality on the ball of radius $1$ of $L_{\frac{p}{p-1}}$ (using the fact that it is separable for $p<\infty$, or that $[0,1]$ is separable for $p=\infty$), and also use Hahn Banach's theorem to obtain
\begin{align*}
\Pr \Big\{ \big| \|\Pi_{V_j} (\hat f_n -f)\|_{p}   - \E [ \|\Pi_{V_{j}}(\hat f_n - f)\|_{p}] \big| \geq 2\sqrt{\log(1/\delta)} \sigma \Big\} \leq \delta,
\end{align*}
where $\sigma^2 \leq \sup_{u: \|u\|_{p/(p-1)} \leq 1}\E [ <\Pi_{V_{j}}(\hat f_n - f), u>^2 ]$. Then we remark that $L_{\frac{p}{p-1}} \subset B_{0,p/(p-1),2}$ (see Proposition~\ref{prop:embeed}), which implies that there exists a universal constant $C_{p/(p-1)}$ such that $\|.\|_{\frac{p}{p-1}} \geq C_{p/(p-1)} \|.\|_{0,\frac{p}{p-1},2}$. This implies in particular that $\sup_{u: \|u\|_{p/(p-1)} \leq 1} \|\Pi_{V_j} u\|_2^2 \leq \sup_{u: \|u\|_{0,p/(p-1),2} \leq C_{p/(p-1)}} \|\Pi_{V_j} u\|_2^2$. This in particular implies, using previous results, that $\sigma^2 \frac{C_{p/(p-1)}^2 2^{j(1 - 2/p)}}{n}$, which leads finally to
\begin{align*}
\Pr \Big\{ \big| \|\Pi_{V_{j}}(\hat f_n - f)\|_{p}   - \E [ \|\Pi_{V_{j}}(\hat f_n - f)\|_{p}] \big| \geq 2C_{p/(p-1)} \sqrt{\log(1/\delta) \frac{ 2^{j(1- 2/p)}}{n}}  \Big\} \leq \delta,
\end{align*}
\end{proof}



\section{Adaptive estimation}\label{sec:estimation}



We prove that adaptive estimators exist on sets that are slightly larger than $\Sigma(r, B)$. As a corollary, adaptive estimators exist on $\Sigma(r, B)$ (Theorem~\ref{th:inference}).

\begin{theorem}\label{th:inferenceersatz}
 There exists an adaptive estimator $\hat f_n(dY^{(n)})$ such that there are two constants $u_p>0$ and $N_p>0$ that depend only on $p$ such that for every $B> 0$, every $r > 0$ and every $c_a \geq 0$, we have
\begin{align*}
\sup_{f\in L_2: \|f - \Sigma(r, B)\|_p \leq c_a  n^{-r/(2r+1)}} \E_f \|\hat f_n - f \|_p &\leq u_p  \Big(\big(c_a+B\big)^{1/(2r+1)} + 1\Big) n^{-r/(2r+1)}.
\end{align*}
We can rewrite this as
\begin{align*}
\sup_{r>0}\sup_{f\in L_2: \|f - \Sigma(r, B)\|_p \leq c_a  n^{-r/(2r+1)}} \Big[ \frac{1}{\tilde U_p(B, c_a)} n^{r/(2r+1)} \E_f \|\hat f_n - f \|_p\Big] &\leq 1,
\end{align*}
where $\tilde U_p(B, c_a) = u_p  \Big((c_a+B) + 2\Big)$.
\end{theorem}

\subsection{Approximation and estimation errors of a thresholded estimator}


The wavelet basis we use is the Cohen-Daubechies-Vial wavelet basis (it that satisfies Assumption~\ref{ass:wav2}).

We first remind the following Corollary of Proposition~\ref{prop:est}
\begin{corollary}\label{cor:est}
Consider $f \in L_2$. There exists a universal constant $D_p$ that depends on $p$ only such that for any fixed $j \in \mathbb N^*$ we have
\begin{equation*}
\E\|\Pi_{V_j} \hat f_n -\Pi_{V_j} f\|_p \leq D_p \frac{2^{j/2}}{n^{1/2}} := \sigma(j,n).
\end{equation*}
\end{corollary}
\begin{proof}
Since $p\geq 2$ (and $p <\infty$), we know by convexity that $\E \|.\|_p^p \geq (\E\|.\|_p)^{1/p}$, which concludes the proof together with Proposition~\ref{prop:est}.
\end{proof}

We state the following Lemma, which is an extension of results in~\citep{hardle1998wavelets, gine2012rates}).
\begin{lemma}\label{lem:approx}
Let $\epsilon>0$. Let $f\in L_2$ such that $\|f-\Sigma(r,B)\|_p \leq \epsilon$. There exists a universal constant $C_p$ that depends on $p$ only such that for any fixed $j \in \mathbb N^*$ and any $f \in B_{r,p,\infty}$ such that $\|f\|_{r, p, \infty} \leq B$, we have
\begin{equation*}
\|f -\Pi_{V_j} f\|_p \leq 2\epsilon + C_p B 2^{-jr}  := B(j,f,\epsilon).
\end{equation*}
\end{lemma}
\begin{proof}
We have
\begin{align*}
\|f - \Pi_{V_j} f\|_p &= \inf_{g \in \Sigma(r,B)}\|f-(g - \Pi_{V_j}g)+(g - \Pi_{V_j}g) - \Pi_{V_j} f\|_p \\
&\leq \inf_{g \in \Sigma(r,B)} \Big[\|f-g\|_p + \|g -\Pi_{V_j} g\|_p +  \|\Pi_{V_j} (f-g)\|_p\Big]\\
&\leq 2\epsilon + C_p \sup_{g \in \Sigma(r,B)}\|g\|_{r,p,\infty} 2^{-jr} \leq 2\epsilon + C_p B 2^{-jr}  := B(j,f,\epsilon).
\end{align*}
\end{proof}

\subsection{Definition of a Lepski type estimator}


Let $c_a>0$, $r>0$ and $B>0$. Let $c>0$ and $f\in L_2$ such that $\|f-\Sigma(r,B)\|_p \leq \epsilon$, where 
\begin{equation*}
\epsilon_n \equiv \epsilon = c_a n^{\frac{-r}{2r+1}}.
\end{equation*}

Set for $D_p' = D_p + 2 C_{p/(p-1)}$
\begin{equation*}
\hat j_n = \min\Big\{ j \in  \mathbb N: \|\Pi_{V_j} \hat f_n - \Pi_{V_l} \hat f_n\|_p \leq 4(D_p'+1) \frac{2^{l/2}}{n^{1/2}}, \forall l>j, l \in \mathbb N \Big\}.
\end{equation*}

We consider in the sequel the adaptive Lepski type estimator $\hat f_n(\hat j_n) = \Pi_{V_{\hat j_n}} \hat f_n$.

Set now the oracle threshold
\begin{equation*}
j^* = j^*(f) =  \inf\Big\{ j \in \mathbb N : B(j,f,\epsilon) \geq \sigma(j,n)\Big\}.
\end{equation*}
Note that by Corollary~\ref{cor:est} and Lemma~\ref{lem:approx} we have
\begin{equation}\label{eq:jstar}
B(j^*,f, \epsilon) \leq \sigma(j^*,n) \leq 2\Big(\frac{1}{D_p} (C_pB + c_a)\Big)^{1/(2r+1)} n^{-r/(2r+1)}.
\end{equation}

\subsection{Bound for the error on the event $\{ \hat j_n \leq j^*\}$}

We have by triangular inequality, Equation~\eqref{eq:jstar} and the definitions of $\hat j_n$ and $j^*$ that (since $D_p'\geq D_p$)
\begin{align*}
\E\Big[ \|f_n(\hat j_n) - f\|_p \ind{\hat j_n \leq j^*} \Big] &\leq \E\Big[ \big(\|\hat f_n(\hat j_n) - \hat f_n(j^*)\|_p + \|\hat f_n(j^*) - f\|_p\big) \ind{\hat j_n \leq j^*} \Big]\\
&\leq 4(D_p'+1) \frac{2^{j^*/2}}{n^{1/2}} + \sigma(j^*,n)\\
&\leq 2\Big(\frac{5D_p'+ 4}{D_p} \Big)\big(\frac{1}{D_p} (C_pB + c_a)\big)^{1/(2r+1)}  n^{-r/(2r+1)}.
\end{align*}

\subsection{Bound for the error on the event $\{ \hat j_n > j^*\}$}

We remind the following Lemma (see~\citep{gine2012rates}).
\begin{lemma}\label{lem:conclp}
There is a constant $D_p'$ that depends on $p$ only and such that
\begin{align*}
\sup_{f \in L_p} \Pr \Big\{  \|\Pi_{V_{j}}(\hat f_n - f)\|_{p}   \geq D_p'\sqrt{\frac{ 2^{j}}{n}}  \Big\} \leq 2^{-2^{2j/p}},
\end{align*}
\end{lemma}
\begin{proof}
Proposition~\ref{prop:Borell} gives us that for the $\|.\|_p$ norm, we have
\begin{align*}
\Pr \Big\{ \big| \|\Pi_{V_{j}}(\hat f_n - f)\|_{p}   - \E [ \|\Pi_{V_{j}}(\hat f_n - f)\|_{p}\big| \geq 2C_{p/(p-1)}\sqrt{\log(1/\delta) \frac{ 2^{j(1- 2/p)}}{n}}  \Big\} \leq \delta.
\end{align*}
By combining this with Corollary~\ref{cor:est}, we get
\begin{align*}
\Pr \Big\{  \|\Pi_{V_{j}}(\hat f_n - f)\|_{p}   \geq D_p\sqrt{\frac{ 2^{j}}{n}}  + 2C_{p/(p-1)}\sqrt{\log(1/\delta) \frac{ 2^{j(1- 2/p)}}{n}}  \Big\} \leq \delta,
\end{align*}
which implies by considering $\delta$ such that $\log(1/\delta)= 2^{2j/p}$
\begin{align*}
\Pr \Big\{  \|\Pi_{V_{j}}(\hat f_n - f)\|_{p}   \geq (D_p+2C_{p/(p-1)})\sqrt{\frac{ 2^{j}}{n}}  \Big\} \leq \exp(-2^{2j/p}).
\end{align*}
\end{proof}

We have by H\"older's inequality that
\begin{align*}
\E\Big[ \|f_n(\hat j_n) - f\|_p \ind{\hat j_n > j^*} \Big] &= \sum_{j>j^*} \E\Big[ \|f_n(j) - f\|_p \ind{\hat j_n = j} \Big] \\
&\leq \sum_{j>j^*} \big( \E \|f_n(j) - f\|_p^p\big)^{1/p} \big(\E \ind{\hat j_n = j}^{p/(p-1)} \big)^{(p-1)/p},
\end{align*}
which implies by Corollary~\ref{cor:est} that we have
\begin{align}\label{eq:probabili}
\E\Big[ \|f_n(\hat j_n) - f\|_p \ind{\hat j_n > j^*} \Big] &\leq \sum_{j>j^*}\sigma(j,n) \Prob(\hat j_n = j)^{(p-1)/p},
\end{align}

By an union bound and by definition of $\hat j_n$, we remark that
\begin{equation}\label{eq:probabili2}
\Prob(\hat j_n = j) \leq \sum_{l \geq j} \Prob\Big( \|\hat f_n(j-1) - \hat f_n(l)\|_p \geq  4(D_p'+1) \frac{2^{l/2}}{n^{1/2}} = \frac{4(D_p'+1)}{D_p} \sigma(l,n) \Big),
\end{equation}
and by triangle inequality we have
\begin{align*}
&\|\hat f_n(j-1) - \hat f_n(l)\|_p\\ 
&= \|\hat f_n(j-1) - \hat f_n(l) - \Pi_{V_{j-1}} f + \Pi_{V_{j-1}} f + \Pi_{V_l} f - \Pi_{V_l} f -f +f\|_p\\
&\leq \|\hat f_n(j-1) - \Pi_{V_{j-1}} f\|_p + \|\hat f_n(l) - \Pi_{V_l} f\|_p + \|\Pi_{V_{j-1}} f - f\|_p +\| \Pi_{V_l} f -f\|_p\\
&\leq \|\hat f_n(j-1) - \Pi_{V_j} f\|_p + \|\hat f_n(l) - \Pi_{V_l} f\|_p + 2 \sigma(l,n),
\end{align*}
since as $l >j-1 \geq j^*$, we have $\|\Pi_{V_{j-1}} f - f\|_p +\| \Pi_{V_l} f -f\|_p \leq B(j-1,f,\epsilon) + B(l,f,\epsilon) \leq 2 B(j^*,f,\epsilon) \leq 2 \sigma(j^*,n) \leq 2 \sigma(l,n)$ by Lemma~\ref{lem:approx}. 
This implies that (since $D_p' \geq D_p$ by definition)
\begin{align*}
&\Prob\Big( \|\hat f_n(j-1) - \hat f_n(l)\|_p \geq \frac{4(D_p'+1)}{D_p} \sigma(l,n) \Big)\\
 &\leq \Prob\Big( \|\hat f_n(j-1) - \Pi_{V_{j-1}} f\|_p  \geq \big(\frac{4(D_p'+1)}{2 D_p} - 1\big) \sigma(l,n) \Big)\\
 &+ \Prob\Big( \|\hat f_n(l) - \Pi_{V_l} f\|_p  \geq \big(\frac{4(D_p'+1)}{2 D_p} - 1\big) \sigma(l,n) \Big)\\
&\leq \Prob\Big( \|\hat f_n(j-1) - \Pi_{V_{j-1}} f\|_p  \geq \big(D_p'+2)\big) \sigma(l,n) \Big)\\
 &+ \Prob\Big( \|\hat f_n(l) - \Pi_{V_l} f\|_p  \geq \big(D_p' + 2\big) \sigma(l,n) \Big)
\end{align*}
so we obtain by Lemma~\ref{lem:conclp} that
\begin{equation*}
\Prob\Big( \|\hat f_n(j-1) - \hat f_n(l)\|_p \geq \frac{4(D_p'+1)}{D_p} \sigma(l,n) \Big) \leq 2\times \exp(-2^{2l/p}),
\end{equation*}
which implies when combined with Equations~\ref{eq:probabili} and~\ref{eq:probabili2}
\begin{align*}
\E\Big[ \|\hat f_n(\hat j_n) - f\|_p \ind{\hat j_n > j^*} \Big] &\leq \sum_{j>j^*} D_p' \sqrt{\frac{2^j}{n}} \sum_{l\geq j}\big(2 \times \exp(-2^{2l/p}) \big)^{(p-1)/p}\\
&\leq \sum_{j>j^*} 8 D_p' \sqrt{\frac{2^j}{n}} \exp(-2^{2j/p}(p-1)/p)\\
&\leq \sum_{j>j^*} 16 D_p' \frac{\exp(-2^{2j/p}(p-1)/2p)}{\sqrt{n}}\\
&\leq 16 D_p' \frac{\exp(-2^{2j^*/p}(p-1)/2p)}{\sqrt{n}}\\
&\leq n^{-r/(2r+1)}.
\end{align*}
for $n$ (and thus $j^*$) large enough (but depending only on $p$, i.e.~$n \geq N_p$).

\subsection{Conclusion}

By combining the results of the two precedent Subsections, we have
\begin{align*}
\E\Big[ \|f_n(\hat j_n) - f\|_p \Big] &\leq  \E\Big[ \|f_n(\hat j_n) - f\|_p \ind{\hat j_n \leq j^*} \Big] +\E\Big[ \|f_n(\hat j_n) - f\|_p \ind{\hat j_n > j^*} \Big]\\
&\leq \Bigg( 2\Big(\frac{5D_p'+ 4}{D_p} \Big)\big(\frac{1}{D_p} (C_pB + c_a)\big)^{1/(2r+1)} + 1\Bigg) n^{-r/(2r+1)}.
\end{align*}
which concludes the proof since all the constants in the bound depend only on $p$.

\section{Extension of Theorem~\ref{th:ssmallseg1} to the entire segment $[t,s]$}\label{sec:adaptiv2}

We now state the analogue of Theorem~\ref{th:bull2}, i.e.~for the whole segment $I=[t,s]$ (still when $s(1-1/p) \leq t$). The proof of this Theorem is more technical than the proof of Theorem~\ref{th:ssmallseg1}, but it is based on similar ideas.

Consider now the case where $s(1-1/p) \leq t$. In this case, full adaptation is possible without constraining the model $\mathcal P_n$ to be a strict subset of $\Sigma(t,B)$. We provide the following result, related to the case $s \leq 2t$ in~\citep{bull2011adaptive}.
\begin{theorem}\label{th:ssmallseg}
Let $1/2\leq t< s$. Assume also that $s(1-1/p) \leq t$. Let $\mathcal P_n = \Sigma(t,B)$ and $I=[t,s]$. Let $B>0$ and $\alpha>0$. There exists a $(L_p,\alpha)-$honest and adaptive confidence set given $\mathcal P_n$, $I$ and $B$.
\end{theorem}
\begin{proof}

Assume that $s(1-1/p) \leq t \leq s$ and let $B>0$. Let $n>\max(N_p, \exp(2/t))$, where $N_p$ defined as in Theorem~\ref{th:inferenceersatz}.

Let $r \in [t,s]$. We define the following sets (similar to the sets defined in Equation~\eqref{eq:tildesets}, but separated from $\Sigma(r,B)$)
\begin{equation*}
\tilde \Sigma(t,r,B, \rho_n(r)) = \Sigma(t,B) \setminus \{g \in \Sigma(t,B): \|g - \Sigma(r,B)\|_{p} \leq \rho_n(r)\}.
\end{equation*}
where $\rho_n(r) \geq 0$. These sets are empty as $r \rightarrow t$, or when $\rho_n(r)$ is large, but are nevertheless defined.

Let $\alpha >0$. Let us write, for $r \in [t,s]$, $\Psi_n(r)$ for the test described in Subsection~\ref{ss:stats2bis}, where the associated constants $E_1$ and $E_2$ are chosen large enough (depending only on $p,B,\alpha$) so that Lemma~\ref{incrifunc} holds. Set also
\begin{equation*}
\rho_n(r)  = 2 C(B+1) n^{-\frac{r}{2r+1}}  \geq C(B+1) (n^{-\frac{r}{2r+1}} + n^{-\frac{t}{2t+1 - 1/p}}),
\end{equation*}
for $C$ as in Theorem~\ref{th:sbig} (depending only on $\alpha$, $p$).

\subsection{Step 1: Study of the process $(\Psi_n(r))_{r\in [t,s]}$}

A first remark is that for any $r\in [t,s]$, the test $\Psi_n(r)$ is a measurable random variable from $(\mathcal C[0,1], \mathcal B(\mathcal C[0,1]))$ to $(\{0,1\}, \big\{\{0\}, \{1\}, \{0,1\}, \emptyset\big\})$ where $\mathcal C[0,1])$ is the set of continuous functions from $[0,1]$ to $\mathbb R$, and $\mathcal B(.)$ is the associated Borel set.



\begin{lemma}\label{incrifunc}
Consider the test $\Psi_n(r)$ described in Subsection~\ref{ss:stats2bis}. Assume that the associated constants $E_1$ and $E_2$ are large enough (depending only on $p,B,\alpha$). The trajectories $r\in [t,s] \rightarrow \Psi_n(r)$ of the process $(\Psi_n(r))_{r \in [t,s]}$ are monotonously increasing, and caglad (left continuous right limit).
\end{lemma}
\begin{proof}
Consider the tests $\Psi_n(r)$ described in Subsection~\ref{ss:stats2bis}. Since $\Psi_n(r)$ is either $1$ or $0$, increasing monotonicity is equivalent to $\forall (r_1, r_2), t \leq r_1 \leq r_2 \leq s$, $\Psi_n(r_2) = 0 \Rightarrow  \Psi_n(r_1) = 0$.

The tests $\Psi_n(r)$ involve the statistics $T_n(l)$ (similar for any $r$), the statistics
\begin{equation*}
\tilde T_n(r)= \inf_{g \in \Sigma(r,B)}\|\Pi_{V_{j_r}} \hat f_n - g\|_{p},
\end{equation*}
where $j_r = \lfloor \log(n^{1/(2r+1)})\rfloor$ is a decreasing function of $r$, the thresholds $(t_n(l,r))_{l}$ that are decreasing functions of $r$, and the threshold $\tilde t_n(r)$ that is a decreasing function of $r$. See Subsection~\ref{ss:stats2bis} for a more complete definition of all these quantities. The test is defined as
\begin{equation*}
 \Psi_n(r)  = 1 -  \ind{\tilde T_n(r) \leq \tilde t_n(r)} \prod_{j_r \leq l\leq j} \ind{T_n(l)\leq(t_n(l,r))^{p}} .
\end{equation*}
Let $t\leq r_1 \leq r_2\leq s$. Assume that $\Psi_n(r_2) = 0$, i.e.~that 
\begin{equation}\label{condi1}
\tilde T_n(r_2) \leq \tilde t_n(r_2),
\end{equation}
and
\begin{equation}\label{condi2}
\forall l\in \mathbb N: j \geq l \geq j_{r_2}, T_n(l)\leq (t_n(l,r_2))^{p}.
\end{equation}

Since $j_{r_1} \geq j_{r_2}$,  and $\forall l, t_n(l,r_1) \geq t_n(l,r_2)$, we know by Equation~\eqref{condi2} that
\begin{equation}\label{condi2.2}
\forall l\in \mathbb N: j \geq l \geq j_{r_1}, T_n(l)\leq (t_n(l,r_1))^{p}.
\end{equation}

If $j_{r_{1}} = j_{r_{2}}$, then $\ind{\tilde T_n(r_1)\leq \tilde t_n(r_1)} = \ind{\tilde T_n(r_2)\leq \tilde t_n(r_2)}=1$. Otherwise, it implies that $j_{r_{1}} \geq 1+ j_{r_{2}}$, and by triangular inequality
\begin{align}
\tilde T_n(r_1) &= \inf_{g \in \Sigma(r_1,B)} \|\Pi_{V_{j_{r_1}}}\hat f_n - g \|_{p} \leq  \inf_{g \in \Sigma(r_2,B)} \|\Pi_{V_{j_{r_1}}}\hat f_n - g \|_{p} \nonumber\\
&\leq  \inf_{g \in \Sigma(r_2,B)} \|\Pi_{V_{j_{r_2}}}\hat f_n - g \|_{p}  +  \sum_{l=j_{r_2}}^{j_{r_1}}\|\Pi_{W_l}\hat f_n \|_{p} \nonumber\\
&\leq \tilde T_n(r_2) + C_p\sum_{l=j_{r_2}}^{j_{r_1}}\|\Pi_{W_l}\hat f_n \|_{0,p,\infty} \nonumber\\
&\leq \tilde T_n(r_2) + C_pE' \sum_{l=j_{r_2}}^{j_{r_1}}(\frac{2^{l/2}}{\sqrt{n}} + \big(\max(T_n(l),0)\big)^{1/p}), \label{eq:gygy}
\end{align}
for some $E'>0$ large enough but depending only on $p$ (see the proof of Lemma~\ref{lem:mean} for the argument on why $\|\Pi_{W_l}\hat f_n \|_{0,p,\infty}^p \leq E''(\frac{2^{lp/2}}{n^{p/2}} + \max(T_n(l),0))$). Since the constants $E_1$ and $E_2$ defined in Subsection~\ref{ss:stats2bis} can be chosen arbitrarily large, and since $s(1-1/p) \leq t$ (which implies that $\tilde t_n(r) \equiv \sum_{l=j_{r}}^{j} t_n(l,r) \equiv n^{-\frac{r}{2r+1}}$) , we can choose $E_1$ and $E_2$ such that 
\begin{equation*}
\tilde t_n(r) =E_2 \sqrt{\frac{2^{j_r}}{n}} \geq E \sum_{l=j_{r}}^{j} t_n(l,r)
\end{equation*}
for some arbitrarily large $E_2>0$, and some arbitrarily large $E>E_2$ (by choosing $E_2/E_1$ large enough). Using this together with Equation~\eqref{eq:gygy}, and the fact that $j_{r_1} \geq j_{r_2}+1$, one obtains by Equations~\eqref{condi1} and~\eqref{condi2}
\begin{align*}
\tilde T_n(r_1) &\leq \tilde t_n(r_2) + C_p E' \sum_{l=j_{r_2}}^{j_{r_1}}(\frac{2^{l/2}}{\sqrt{n}} + t_n(l,r_2))\\
&\leq 4C_p E' \sqrt{\frac{2^{j_{r_1}}}{n}} + (E_2 + C_pE'E_2/E) \sqrt{\frac{2^{j_{r_2}}}{n}} \\
&\leq (E_2/\sqrt{2} + C_pE'E_2/(\sqrt{2}E) + 4C_p E')  \sqrt{\frac{2^{j_{r_1}}}{n}} \\
&\leq E_2  \sqrt{\frac{2^{j_{r_1}}}{n}} = \tilde t_n(r_1),
\end{align*}
for $E_2$ and $E/E_2$ large enough. This implies together with Equation~\eqref{condi2.2} that
\begin{equation*}
 \Psi_n(r_1)  = 1 -  \ind{\tilde T_n(r_1) \leq \tilde t_n(r_1)} \prod_{j_{r_1} \leq l\leq j} \ind{T_n(l)\leq(t_n(l,r_1))^{p}}  = 0.
\end{equation*}
This concludes the proof of increasing monotonicity.

The trajectories $r \rightarrow \Psi_n(r)$ are increasing in $\{0,1\}$. They are thus either caglad, or cadlag. By definition of the test, the sets $\{r\in [t,s]: \Psi_n(r) = 1\}$ are closed subsets of $[t,s]$. The trajectories are thus caglad.
\end{proof}

Lemma~\ref{incrifunc}, together with the fact that $\Psi_n(r)$ is measurable for any $r$, implies that the process $(\Psi_n(r))_{r \in [t,s]}$ is progressively measurable.

\subsection{Step 2: Estimation of the Besov exponent}


Consider $f\in \Sigma(t,B)$. As stated in Lemma~\ref{incrifunc}, the trajectories $r \rightarrow \Psi_n(r)$ are increasing functions. More precisely, their value is $0$ until some value $\hat r$ defined as
\begin{equation*}
 \hat r = \inf \Big\{r\in [t,s]: \Psi_n(r) = 1 \Big\},
\end{equation*}
and then $1$ for $r$ large enough. $\hat r \in [t,s]$ is well defined since the trajectories are bounded by $1$, and measurable since it is a stopping time on the progressively measurable process $(\Psi_n(r))_{r \in [t,s]}$ with caglad trajectories. Note also that, since the trajectories $\Psi_n(r)$ are of the form $x\in [t,s] \rightarrow \mathbf 1 \{x > c \}$
\begin{equation*}
 \hat r = \inf \Big\{r\in [t,s]: \Psi_n(r) = 1 \Big\} = \sup \Big\{r\in [t,s]: \Psi_n(r) = 0 \Big\}.
\end{equation*}
Consider the confidence set around $\hat f_n(\hat j_n)$, which is the adaptive estimate considered in Theorem~\ref{th:inference}, as being 
\begin{equation*}
C_n = \Big\{g: \|\hat f_n(\hat j_n) - g\|_p \leq  \frac{1}{\alpha}U_p' n^{-\frac{\hat r}{2\hat r+1}} \Big\},
\end{equation*}
where $U_p' = \tilde U_p(B,2C(B+1))$ is defined as in Theorem~\ref{th:inferenceersatz}. Note that $U_p'$ depends only on $B$ and $C$, and thus only on $B, \alpha, p$.

Write
\begin{equation*}
 r_f = \sup \Big\{r\in [t,s]: \|f - \Sigma(r, B)\|_p = 0 \Big\} = \sup \Big\{r\in [t,s]: f \in \Sigma(r, B) \Big\},
\end{equation*}
for the Besov exponent of $f$, and
\begin{equation*}
 r_f^+ \equiv r_f^+(n) = \sup \Big\{r\in [t,s]: \|f - \Sigma(r, B)\|_p \leq \rho_n(r) \Big\}.
\end{equation*}
Note that $r_f^+ $ exists since $\|f - \Sigma(t, B)\|_p=0$.

Since $r \in [t,s] \rightarrow \|f - \Sigma(r, B)\|_p$ is a monotonously increasing function in $r$, we know that $\forall \epsilon>0$, 
\begin{equation}\label{eq:cac}
f \in \Sigma(r_f-\epsilon,B).
\end{equation}
Also,  for the same reason and since $\rho_n(r)$ is a decreasing function of $r$,
\begin{equation}\label{eq:cac2}
\|f - \Sigma(r_f^++\epsilon ,B)\|_2 \geq \rho_n(r_f^+) \geq \rho_n( r_f^+ + \epsilon),
\end{equation}
and
\begin{equation}\label{eq:cac3}
\|f - \Sigma(r_f^+-\epsilon ,B)\|_2 \leq \rho_n(r_f^+) \leq \rho_n( r_f^+ -  \epsilon),
\end{equation}
where by convention, $\|f - \emptyset\|_2 = \|f\|_2$. We set 
\begin{equation*}
\epsilon \equiv \epsilon_n = 1/\log(n).
\end{equation*}

Since $f \in \Sigma(r_f-\epsilon ,B)$ (Equation~\eqref{eq:cac}), we have in particular by Equation~\eqref{eq:consis} that for the $n$ we fixed
\begin{equation}\label{eq:highi}
{\Pr}_f(\Psi_n(r_f - \epsilon) = 0)\geq \inf_{g \in \Sigma(r_f-\epsilon, B)}{\Pr}_g(\Psi_n(r_f - \epsilon) = 0) \geq 1-\alpha.
\end{equation}
Since $\|f - \Sigma(r_f^+ + \epsilon ,B)\|_2 \geq \rho_n(r_f^+ + \epsilon )$ (Equation~\eqref{eq:cac2}), and since thus $f \in \tilde \Sigma(t, r_f^+ + \epsilon ,B, \rho_n(r_f^+ + \epsilon ))$, we have in particular by Equation~\eqref{eq:consis} that for the $n$ we fixed
\begin{equation}\label{eq:lowi}
{\Pr}_f(\Psi_n(r_f^+ + \epsilon) = 1) \geq \inf_{g \in \tilde \Sigma(t, r_f^+ + \epsilon ,B, \rho_n(r_f^+ + \epsilon ))}{\Pr}_g(\Psi_n(r_f^+ + \epsilon) = 1) \geq 1-\alpha.
\end{equation}
By combining Equations~\eqref{eq:lowi} and~\eqref{eq:highi}, and since $r \rightarrow \Psi_n(r)$ is an increasing function (Lemma~\ref{incrifunc}), we know that
\begin{equation*}
{\Pr}_f ( \hat r \in [r_f - \epsilon , r_f^+ + \epsilon]) \geq 1-2\alpha.
\end{equation*}

\subsection{Step 3: Bound on the diameter of the confidence set}

The bound of last Equation holds for any $f \in \Sigma(t,B)$ for the $n$ we fixed, and thus by just considering the infimum over $\Sigma(t,B)$, we have
\begin{equation*}
\inf_{f \in \Sigma(t,B)}{\Pr}_f ( \hat r \in [r_f - \epsilon , r_f^+ + \epsilon ]) \geq 1-2\alpha.
\end{equation*}
We thus have by definition of $C_n$ that
\begin{equation*}
\sup_{f \in \Sigma(t,B)}{\Pr}_{f}\Big(|C_n| > \frac{1}{\alpha}U_p' n^{-\frac{r_f - \epsilon}{2(r_f- \epsilon)+1}}\Big)  \leq 1-  \inf_{f \in \Sigma(t,B)} {\Pr}_f ( \hat r \in [r_f - \epsilon, r_f^+  + \epsilon])  \leq 2\alpha,
\end{equation*}
and since $\epsilon = 1/\log(n)$, this implies for the $n$ we fixed
\begin{equation*}
\sup_{f \in \Sigma(t,B)}{\Pr}_{f}\Big(|C_n| > \frac{1}{\alpha}U_p' \exp(1) n^{-\frac{r_f}{2r_f+1}}\Big)  \leq 1-  \inf_{f \in \Sigma(t,B)} {\Pr}_f ( \hat r \in [r_f - \epsilon, r_f^+  + \epsilon])  \leq 2\alpha,
\end{equation*}
since $\exp(1/(2(r_f - \epsilon) + 1)) \leq \exp(1)$. Note now that
\begin{align*}
\sup_{f \in \Sigma(t,B)}{\Pr}_{f}\Big(|C_n| > \frac{U_p'}{\alpha} \exp(1) n^{-\frac{r_f}{2r_f+1}}\Big)  &=  \sup_{r\in [t,s]}\sup_{f \in \Sigma(r,B)}{\Pr}_{f}\Big(|C_n| > \frac{U_p'}{\alpha} \exp(1) n^{-\frac{r_f}{2r_f+1}}\Big)\\
&\geq \sup_{r\in [t,s]}\sup_{f \in \Sigma(r,B)}{\Pr}_{f}\Big(|C_n| > \frac{U_p'}{\alpha} \exp(1) n^{-\frac{r}{2r+1}}\Big),
\end{align*}
since $f \in \Sigma(r,B)$ implies $r \leq r_f$. Finally
\begin{equation*}
\sup_{r\in [t,s]} \sup_{f \in \Sigma(r,B)}{\Pr}_{f}\Big(|C_n| > \frac{1}{\alpha}U_p'\exp(1) n^{-\frac{r}{2r+1}}\Big) \leq 2\alpha.
\end{equation*}

\subsection{Step 4: Bound on the probability that the parameter is in the confidence set}

Also we have by Markov's inequality
\begin{align}
  \inf_{f \in \Sigma(t,B)} {\Pr}_{f}\Big( f \in C_n \Big) &\geq 1 - \sup_{f \in \Sigma(t,B)} {\Pr}_{f}\Big( \|\hat f_n(\hat j_n) - f\|_p \geq \frac{1}{\alpha}U_p' n^{-\frac{\hat r}{2\hat r +1}}\Big) \nonumber\\
&\geq 1-  \sup_{f \in \Sigma(t,B)} {\Pr}_f(\hat r \not\in [r_f- \epsilon,r_f^+ + \epsilon]) \nonumber\\
&- \sup_{f \in \Sigma(t,B)} {\Pr}_{f}\Big( \|\hat f_n(\hat j_n) - f\|_p \geq  \frac{1}{\alpha}U_p' n^{-\frac{r_f^+ + \epsilon}{2(r_f^+  + \epsilon) +1}}\Big) \nonumber\\
&\geq 1- 2\alpha  -  \sup_{f \in \Sigma(t,B)} {\Pr}_{f}\Big( \|\hat f_n(\hat j_n) - f\|_p \geq  \frac{1}{\alpha}U_p' \exp(-2) n^{-\frac{r_f^+ - \epsilon}{2(r_f^+ - \epsilon) +1}}\Big) \nonumber\\
&\geq 1- 2\alpha  -  \sup_{f \in \Sigma(t,B)} \Big[ \frac{\alpha}{U_p'} \exp(2) n^{\frac{r_f^+ - \epsilon}{2(r_f^+ - \epsilon) +1}} \mathbb E_{f} \|\hat f_n(\hat j_n) - f\|_p  \Big], \label{eq:lm}
\end{align}
by definition of $\epsilon$ and since $\exp(- 2/(2(r_f + \epsilon) + 1)) \geq \exp(-2)$. 

We have for any $f \in \Sigma(t,B)$ that 
\begin{equation*}
\|f - \Sigma(r_f^+ - \epsilon,B)\|_p \leq \rho_n(r_f^+ - \epsilon) = 2C(B+1)n^{-(r_f^+ - \epsilon)/(2(r_f^+ - \epsilon) +1)}
\end{equation*}
by Equation~\eqref{eq:cac3}. Since $r_f^+ - \epsilon>t-1/\log(n)>t/2>0$ by definition of $n$, we have
\begin{align*}
&\sup_{f \in \Sigma(t,B)} \Big[ \frac{1}{\tilde U_p(B,2C(B+1))} n^{\frac{r_f^+ - \epsilon}{2(r_f^+ - \epsilon) +1}} \E_f \|\hat f_n(\hat j_n) - f \|_p \Big]\\
&\leq \sup_{r>0}\sup_{f \in L_2: \|f - \Sigma(r,B)\|_2 \leq \rho_n(r)} \Big[\frac{1}{  \tilde U_p(B,2C(B+1))} n^{r/(2r+1)}  \E_f \|\hat f_n(\hat j_n) - f \|_p \Big].
\end{align*}
Combining this with Theorem~\ref{th:inferenceersatz} implies, since $n>N_p$, that
\begin{align*}
&\sup_{f \in \Sigma(t,B)} \Big[ \frac{1}{\tilde U_p(B,2C(B+1))} n^{\frac{r_f^+ - \epsilon}{2(r_f^+ - \epsilon) +1}} \E_f \|\hat f_n(\hat j_n) - f \|_p \Big] \leq 1,
\end{align*}
since $U_p' \geq \tilde U_p(B,2C(B+1))$. 
We conclude by plugging this result into Equation~\eqref{eq:lm} that
\begin{align*}
\inf_{f \in \Sigma(t,B)} {\Pr}_{f}\Big( f \in C_n \Big)
&\geq 1- 3\exp(2)\alpha.
\end{align*}

\paragraph{Conclusion.}
All these results hold for \textit{any} $n>\max(N_p, \exp(2/t))$. We have thus proven that $C_n$ is an honest and adaptive confidence set on $\Sigma(t,B)$ for the whole interval $[t,s]$ and $B$.
\end{proof}

\section{Discussion on the extension of the results to more general settings}\label{app:discnongauss}

The test statistic $\Psi_n$ that is considered relies mostly on estimates of $|a_{l,k}|^p$, which are the $\hat F_{p}^p(l,k)$. The obstacle for generalising the method presented in the paper to the regression setting, is the adaptation of these estimates to the regression setting. But the construction of the $\hat F_{p}^p(l,k)$ depends crucially on the distribution of the noise (error) to the signal $f$. Indeed, in the computation of the quantities $\hat F_p^p(l,k)$, we plug the $p$ first moments of a Gaussian distribution in order to correct the bias of $|\hat a_{l,k}|^p$ toward $|a_{l,k}|^p$. If the distribution of the noise is not Gaussian, the bias is not going to be corrected by these (Gaussian) moments, and we would want to replace them with the moments of the noise. 







However, if we do not wish to assume that we know the distribution of $\xi$ (or, moreover, if it is heterocedastic) the construction of $L_p$-adaptive and honest confidence sets in this setting is possible but slightly different from the construction proposed. We did not present it in the paper since it is rather technical but not fundamentally different from what happens in the Gaussian process setting.


We first remind how to estimate $\hat a_{l,k}$ in two classic settings, i.e.~regression and density estimation. In the setting of density estimation, i.e.~the data in this setting is $n$ i.i.d.~samples from a random variable of density $f$, we can estimate $a_{l,k}$ by 
$$\hat a_{l,k} =\frac{1}{n} \sum_{j\leq n} \psi_{l,k}(X_j).$$
If the density $f$ is bounded (and still defined on the compact $[0,1]$) then the estimates $\hat a_{l,k}$ computed in this way will be unbiased and have a variance-covariance structure that is of same order than in the case of the Gaussian process model.
In the setting of non-parametric regression, i.e.~the data in this setting is $n$ i.i.d.~samples $(X_j, Y_j)_n$ such that $Y_j = f(X_j) + \xi_j(X_j)$ where $\xi_j(X_j)$ is the noise, we can estimate $a_{l,k}$ by 
$$\hat a_{l,k} =\frac{1}{n}  \sum_{j\leq n} \psi_{l,k}(X_j) Y_j.$$
If the function $f$ is bounded (and still defined on the compact $[0,1]$), the design $X_j$ is uniformly random on $[0,1]$ and the noise $\xi_j(X_j)$ is independent in $j$ (although it might depend on $X_j$), of mean $0$, and sub-Gaussian, then the estimates $\hat a_{l,k}$ computed in this way will be unbiased and have a variance-covariance structure that is of same order than in the case of the Gaussian process model.


Now, a first idea to adapt $\hat F_p^p(l,k)$ to these settings (if we have $2n$ data) is to divide the sample in two sub-samples of size $n$, estimate the $p$ first moments of the distribution of $\sqrt{n} \hat a_{l,k}$ on the first half (that we write $\hat m_u^u$), compute $\hat a_{l,k}$ on the second sub-sample, and then redefine the $\hat F_{p}^p(l,k)$ as
\begin{align*}
\hat F_{p}^p(l,k) = |\hat a_{l,k}|^{p} - \sum_{u=0, u \hspace{1mm} even}^{\lfloor p \rfloor - 2} \mathbf  C_p^u \big(\frac{\hat m_{p-u}}{n^{1/2}}\big)^{(p-u)}  \hat F_{u}^u(l,k),
\end{align*}
where
\begin{align*}
\hat F_{u}^u(l,k) = \hat a_{l,k}^u - \sum_{i=0, i \hspace{1mm} even}^{u-2} \mathbf C_u^i \big(\frac{\hat m_{u-i}}{n^{1/2}}\big)^{(u-i)} \hat F_{i}^i(l,k).
\end{align*}
These quantities will verify the same properties as the $hat F_{p}^p(l,k)$ analysed in the paper (but the proof is more technical).

Another idea is to redefine the estimates $\hat F_p^p(l,k)$ of $|a_{l,k}|^p$. The idea is to divide the data in $\lfloor p \rfloor + \mathbf 1\{p - \lfloor p \rfloor \neq 0\}$ sub-samples of equal size, and to compute in each of these samples estimates of $a_{l,k}$ as described above. Let us denote by $\hat a_{l,k}^{(i)}$ the estimate of $a_{l,k}$ computed with the $i$th sub-sample. We propose to redefine the estimate $\hat F_p^p(l,k)$ of $|a_{l,k}|^p$ as
\begin{equation}\label{eq:alter}
\hat F_p^p(l,k)= \prod_{i=1}^{\lfloor p \rfloor}  \hat a_{l,k}^{(i)} \times \Big| \hat a_{l,k}^{(\lfloor p \rfloor + \mathbf 1\{p - \lfloor p \rfloor \neq 0})\} \Big|^{p - \lfloor p \rfloor}.
\end{equation}
The mean and variance of this estimate will verify the same inequalities as the estimate defined in the proof of Theorem 3.6 (Section 5), and similar results will hold. 

\end{document}

%% file: symbolsdef.tex
\newcommand{\ind}[1]{\mathbb I\left\lbrace {#1} \right\rbrace}

\newcommand{\Prob}{\mathbb P}

\newtheorem{assumption}{Assumption}
\newtheorem{proposition}{Proposition}